\documentclass{scrartcl}

\usepackage{amsmath}
\usepackage{amsthm}
\usepackage{amsxtra}
\usepackage{amssymb}

\usepackage[english]{babel}

\usepackage[utf8]{inputenc}

\usepackage[format=plain, labelfont=bf]{caption}
\usepackage{csquotes}
 \usepackage{xcolor}

\usepackage{dsfont}

\usepackage{exscale}

\usepackage{float}

\usepackage{graphicx}
\graphicspath{ {./figures/} }

\usepackage{hyperref}

\usepackage{latexsym}

\usepackage{mathrsfs}
\usepackage{mathtools}
\usepackage{microtype}

\usepackage[automark]{scrlayer-scrpage}
\usepackage{setspace}
\usepackage{subfigure}

\usepackage{tabularx}
\usepackage{todonotes}

\usepackage{verbatim}

\hypersetup{
	colorlinks=false,
}

\newcommand{\IR}{\mathds{R}}

\newcommand{\B}{\mathds{B}}

\newcommand{\IP}{\mathds{P}}

\newcommand{\Z}{\mathds{Z}}
\newcommand{\IE}{\mathds{E}}

\newcommand{\N}{\mathds{N}}
\newcommand{\T}{\mathds{T}}

\newcommand{\II}{\mathbb{I}}
\newcommand{\ID}{\mathbb{D}}
\newcommand{\IU}{\mathbb{U}}

\newcommand{\cC}{\mathcal{C}}
\newcommand{\cD}{\mathcal{D}}

\newcommand{\cI}{\mathcal{I}}

\newcommand{\cG}{\mathcal{G}}

\newcommand{\cM}{\mathcal{M}}

\newcommand{\cN}{\mathcal{N}}

\newcommand{\cS}{\mathcal{S}}
\newcommand{\cR}{\mathcal{R}}

\newcommand{\cT}{\mathcal{T}}
\newcommand{\cU}{\mathcal{U}}

\newcommand{\bfX}{\mathbf{X}}

\newcommand{\bfZ}{\mathbf{Z}}
\newcommand{\bfeta}{\boldsymbol{\eta}}
\newcommand{\bfxi}{\boldsymbol{\xi}}
\newcommand{\bfzeta}{\boldsymbol{\zeta}}

\newcommand{\birth}{b}
\newcommand{\death}{d}

\newcommand{\del}{\partial}

\newcommand{\zero}{o}
\newcommand{\szero}{\mathbf{0}}
\newcommand{\1}{\mathds{1}}
\DeclareMathOperator{\supp}{supp}

\theoremstyle{definition}
\newtheorem{definition}{Definition}[section]
\newtheorem{example}[definition]{Example}
\newtheorem{remark}[definition]{Remark}

\theoremstyle{plain}
\newtheorem{theorem}[definition]{Theorem}
\newtheorem{lemma}[definition]{Lemma}
\newtheorem{corollary}[definition]{Corollary}

\newtheorem{proposition}[definition]{Proposition}

\newtheorem{assumption}[definition]{Assumption}
\makeatletter
\renewenvironment{proof}[1][\proofname]{%
	\par\pushQED{\qed}\normalfont%
	\topsep6\p@\@plus6\p@\relax
	\trivlist\item[\hskip\labelsep\bfseries#1\@addpunct{.}]%
	\ignorespaces
}{%
	\popQED\endtrivlist\@endpefalse
}
\makeatother

\title{Contact process with viral load}
\author{\textsc{Marco Seiler}\footnote{Goethe University Frankfurt, Robert-Mayer-Straße 10, 60486 Frankfurt am Main, Germany, \texttt{seiler@math.uni-frankfurt.de}}}

\begin{document}

\date{\today}
\maketitle

\begin{abstract}
    In this article, we present two novel variants of the contact process. In the first variant individuals carry a viral load. An individual with viral load zero is classified as healthy and otherwise infected. If an individual becomes infected it begins with a viral load of one, which then evolves according to a Birth-Death process. In this model, viral load indicates severity of the infection such that individuals with a higher load can be more infectious. Moreover, the recovery times of individual is not necessarily exponentially distributed and can even be chosen to follow a power-law distribution.
   	In the second variant individuals are permanently infected albeit in two states: actively infected or dormant. The dynamics of these individual states are again governed by a Birth-Death process. Dormant infections do not interact with neighbouring individuals but may reactivate spontaneously. Active infections reactivate dormant neighbours at a constant rate and may become dormant themselves.

   	We present for both variants a Poisson construction. For the first model, we study the phase transition of survival and discuss existence of a non-trivial upper invariant law. Additionally, we derive a duality relationship between the two variant, which we use to uncover a phase transition regarding invariant distributions in the second variant.
\end{abstract}
{\footnotesize\textit{2020 Mathematics Subject Classification --} Primary 60K35
	}
\section{Introduction}

The classical contact process introduced by Harris~\cite{harris1974contact} is a well-studied model, utilised, for example, to model the spread of disease in a spatially structured population typically represented by a graph. In this model, individuals are either infected, transmitting this infection to healthy neighbours at a certain rate $\lambda$, or they recover at a rate of 1 and become healthy. Although the dynamics of the model may seem quite simplistic, several interesting results and phenomena can be observed. A comprehensive overview of classical results can be found in the two books by Liggett~\cite{liggett1985interacting} and \cite{liggett1999stochastic}. More recent developments concerning contact processes on random graphs are discussed in the book by Valesin~\cite{valesin2024contact}.

Two aspects of this model that can be criticised for their lack of realism, depending on the application, are that individuals are strictly categorised as either infected or healthy, and that recovery times are assumed to be exponentially distributed. It is plausible to expect that not all infected individuals are equally infectious. Additionally, empirical data suggest that the incubation periods for certain diseases do not follow an exponential distribution. There is a substantial body of empirical research on this topic; for example, Lover et al.~\cite{Lover2014distribution} provide insights in the context of malaria. 

From a modelling perspective, the assumption of exponential waiting times ensures that the contact process is a Markov process, making it analytically and computationally tractable. To address the aforementioned limitations, we introduce a variant of the contact process in this article, where the state of an individual is indicated by a viral load instead. When an individual becomes infected, a viral load of 1 is assigned, which evolves according to a Birth-Death process. This individual remains infected as long as the viral load is positive and becomes healthy when the viral load reaches 0. One inspiration for this modelling choice comes from bacterial infections, where bacteria can multiply by cell division (i.e., split into two new copies) or die.

This mechanism allows us to let the infection rate $\Lambda(n)$ depend on the viral load $n$ of an individual, where $\Lambda(n)$ will be a non-decreasing function. Thus, the viral load can be viewed as an indication of how infectious an individual is. Furthermore, in this model, an individual recovers when the viral load reaches 0. Consequently, the distribution of recovery times is determined by the extinction time of a Birth-Death process. We will see that, depending on the choice of rate functions, this framework includes the exponential distribution but also accommodates a wider class of distributions while retaining the Markov property of the system, such as power-law distributions.

We provide sufficient conditions for when the process has a positive survival probability, meaning that an infection persists for all time or goes extinct almost surely (see Proposition 2.4 and Theorem 2.5). Additionally, we offer a criterion for the existence of an upper invariant law in Theorem 2.6.

For the special case of constant infection rates, we show in Proposition 2.8 that there exists a duality relation with another variant of the contact process introduced in this article, which we call the contact process with lingering infections. In this model, individuals are never truly healthy; rather, they are either actively infected or the infection is in a dormant state. This dormant state is again modelled by a Birth-Death process, similar to the viral load, indicating how deeply dormant an infection or infected individual is. A dormant infection can re-emerge either through infection by an actively infected neighbour or if the internal birth-death process reaches $0$. There are several viral infections, such as the herpes simplex virus (HSV), that never fully heal and may reoccur frequently later, which could be modelled using this framework.

Finally, we use the duality relation to demonstrate in Theorem 2.10 the existence of a phase transition in this model and identify the corresponding critical value with the critical value for survival of the contact process with viral load. 
\section{Models and main results}
Let $G=\Gamma(\mathfrak{g},\mathfrak{s})$ be a Cayley graph, where the group $\mathfrak{g}$ is generated by the generating set $\mathfrak{s}$ which we assume to be finite and symmetric. As usual, $V$ denotes the set of all vertices and $E$ the (undirected) edge set. In particular, the graph $G=(V,E)$ is connected, transitive and has bounded degrees. This class of graphs contains, in particular, the integer lattices and regular trees. 

We denote by $\zero \in V$ a designated vertex, which we call the origin. Furthermore, $\cN_x:=\{y\in V: \{x,y\}\in E\}$ is the neighbourhood of $x$ and by transitivity follows that $D:=|\cN_{\zero}|=|\cN_x|$ for all $x\in V$, where $|A|$ of a set $A$ denotes it cardinality. Moreover, we denote by $r(x,y)$ the graph distance of two vertices $x,y\in V$.

Since we will consider $\N_0^V$ as state space, we denote by $\szero$ the all zero configuration, i.e.\ $\szero(x)=0$ for all $x\in V$. In addition, we set $\delta_{A}$ as the configuration such that $\delta_{A}(x)=1$ if $x\in A$ and $\delta_{A}(x)=0$ otherwise. If $A=\{x\}$, then we will write $\delta_{x}$ instead of $\delta_{\{x\}}$.  Finally, $|\eta|:=\sum_{x\in V}\eta(x)$ for all $\eta\in \N_0^V$.

We denote by $\overline{\N}_0=\N_0\cup\{\infty\}$ the usual one-point compactification. Sometimes it will be more convenient to consider the compactification $\overline{\N}_0^V$ of our state space instead. As usual we equip $\overline{\N}_0^V$ with the product topology. Then let $(\nu_n)_{n\geq 0}$ and $\nu$ be measures on $\overline{\N}_0^V$ with respect to the Borel $\sigma$-algebra. We say $\nu_n\Rightarrow \nu$ on $\overline{\N}_0^V$, i.e.\ convergences weakly, if $\int f(\eta)\nu_n(d\eta)\to \int f(\eta)\nu(d\eta)$ as $n\to\infty$ for all continuous functions $f\in C\big(\overline{\N}_0^V\big)$.
\subsection{The contact process with viral load}
The \emph{contact process with viral load} $(\bfeta_t)_{t\geq0}$, which we abbreviate with CPVL, is a Markov process, which takes values in $\N_0^V$.
We call $\bfeta_t(x)\in \N_0$ the \emph{virus load} of $x\in V$ at time $t$ and if $\bfeta_t(x)>0$ we call $x$ \emph{infected} and otherwise \emph{healthy}. Let $\Lambda,\birth,\death:\N_0\to [0,\infty)$ with $\death(0)=\birth(0)=0$. Now if $\bfeta_t$ is in state $\eta$, then we have for every $x$ the transitions
\begin{equation}\label{CPViral}
	\begin{aligned} 
		\eta(x)&\to \eta(x)\vee 1\quad \,\,\, \text{at rate } \sum_{y\in \cN_x}\Lambda(\eta(y))\1_{\{\eta(y)>0\}},\\
		\eta(x)&\to \eta(x)+1\quad \text{ at rate } \birth(\eta(x))
		\text{ and }\\
		\eta(x)&\to \eta(x)-1\quad \text{ at rate } \death(\eta(x)).
	\end{aligned}
\end{equation}
where $\Lambda(\,\cdot\,)$ is called the \emph{infection rate} function and $\vee$ denotes the maximum operation. 

We indicate the initial configuration $\eta$ by adding a superscript, that is $(\bfeta_t^{\eta})_{t\geq 0}$. We also adopt the convention that if $\eta=\delta_{x}$ is the initial configuration, then we indicate this by $(\bfeta_t^{x})_{t\geq 0}$.

Next, we impose some assumptions on the rate functions. For that we introduce the Birth-Death (BD) process $X=(X_t)_{t\geq0}$ with initial value $X_0=1$. This means that if $X_t$ is in state $n$, it has transitions 
\begin{equation*}
	\begin{aligned}
		n\to n+1 &\text{ with rate } \birth(n),\\
		n\to n-1 &\text{ with rate } \death(n).
	\end{aligned}    
\end{equation*}
Recall that $\supp(f):=\{n\in \N_0: f(n)>0\}$ is the support of a function $f:\N_0\to \N_0$.

\begin{assumption}\label{ass:RateAssumption} Let $\Lambda,\birth,\death:\N_0\to [0,\infty)$ be non-constant functions with $\death(0)=\birth(0)=0$, which satisfy the following:
	\begin{itemize}
		\item  $\supp(\death(\,\cdot\,+1))=\supp(\birth)\cup \{0\}$ and there either exists a $K\in \N$ such that $\supp(b)=\{1,\dots,K\}$ or $\supp(b)=\N$.
		\item $\birth, \death\in O(n)$ as $n \to \infty$, i.e.\ $\birth$ and $\death$ are at most of linear growth.
		\item $\Lambda$ is non-decreasing and $\IE[\Lambda(X_t)]<\infty$ for all $t\geq 0$.
	\end{itemize}
\end{assumption}
Throughout the rest of the paper, we assume that the rate functions $\Lambda,\birth,\death:\N_0\to \N_0$ satisfy Assumption~\ref{ass:RateAssumption}. Note that the assumption that $\birth, \death\in O(n)$ is necessary, as it ensures non-explosiveness of both processes we are considering. The assumption that $\Lambda$ is non-decreasing ensures that the CPVL is monotone; however, the model would still be well-defined without it.

Let us denote by $C(\N_0^V)$ the Banach space of all bounded continuous functions $f:\N_0^V\to \IR$ equipped with the supremum norm $||f||_{\infty}=\sup_{\eta\in\N^V_0}|f(\eta)|$. Furthermore, we define
	\begin{equation*}
		\cC_{\text{fin}}(\N_0^V):=\{f\in C(\N_0^V): f \text{ only depend on finitely many coordinates}\}.
	\end{equation*}
	We construct the CPVL $\bfeta=(\bfeta_t)_{t\geq 0}$ explicitly via a Poisson construction, so that we obtain an existence result in the following sense.
\begin{theorem}\label{thm:ExistenceOfCPVL}
		Let $\Lambda,\birth,\death:\N_0\to [0,\infty)$ satisfy Assumption~\ref{ass:RateAssumption}. Then, there exists a Feller process $\bfeta=(\bfeta_t)_{t\geq 0}$ with values in $\N_0^V$, whose generator agrees with the operator associated to the transitions \eqref{CPViral} on $\cC_{\text{fin}}(\N_0^V)$. Moreover, the processes $\bfeta^{\eta}=(\bfeta^{\eta}_t)_{t\geq 0}$ with $\bfeta^{\eta}_{0}=\eta$ are defined on the same probability space for all $\eta\in\N_0^V$.
\end{theorem}
This Poisson construction not only provides existence of the process $\bfeta$, but allows us to show that $\bfeta$ has the following properties:
\begin{enumerate}
		\item Suppose $\Lambda',\birth',\death':\N_0\to \N_0$ satisfy Assumption~\ref{ass:RateAssumption}, and $\Lambda'\geq \Lambda$, $\birth'\geq \birth$  and $\death'\leq \death$. Then, there exists a CPVL $\bfeta'$ with infection rate $\Lambda'$ and rate functions $b'$ and $d'$ such that $\bfeta'_t\geq \bfeta_t$ for all $t\geq 0$.
		\item The CPVL $\bfeta$ is additive, i.e.\ $\bfeta_t^{\eta_1\vee\eta_2}= \bfeta_t^{\eta_1} \vee\bfeta_t^{\eta_2}$ for all $t\geq 0$ with $\eta_1,\eta_2\in \N^V$.
		\item The CPVL $\bfeta$ is a monotone Markov process. This means that if $\eta_1,\eta_2\in \N^V$ with $\eta_1\leq  \eta_2$, then $\bfeta_t^{\eta_1}\leq  \bfeta_t^{\eta_2}$ for all $t\geq 0$.
	\end{enumerate}
	\begin{remark}
		Note that, even though we show that $\bfeta$ is a Feller process, we will not prove that the operator associated to the transitions \eqref{CPViral} is indeed the generator of $\bfeta$, since we do not use the explicit form of the generator. In fact, this does not follow from a direct application of the standard results found in \cite{liggett1985interacting} or \cite{swart2022course}, since the state space $\N_0^V$ is not compact, and the rate functions $\birth$, $\death$ and $\Lambda$ are possibly unbounded.  	
	\end{remark}
We denote the first hitting-time of $0$ of the process $X$ by
\begin{equation*}
	\tau_{\text{rec}}:=\inf\{t\geq 0: X_t=0\}.
\end{equation*}
After being infected, the viral load of a vertex $x$ performs the dynamics of a BD-process with the same law as $X$. Thus, the recovery time of the said vertex $x$ has the same distribution as $\tau_{\text{rec}}$, i.e.\
\begin{equation*}
	\tau_{\text{rec}}\stackrel{d}{=} \inf\{t>0: \bfeta^{\zero}_t(\zero)=0\}.
\end{equation*}
This means, we can model the spread of an infection in a population where the recovery times have the same distribution as the extinction time as a continuous time branching process. In the following two examples, we will see that this, for example, includes distributions with exponential as well as power-law tails.   
\begin{example}\label{ex:BirthDeathProcess}
	Choose $\birth(n)=(n+(1-a)_{+})\1_{\{n>0\}}$ and $\death(n)=(n+(a-1)_{+})\1_{\{n>0\}}$ with $a>0$ and $(\,\cdot\,)_{+}=\max\{\,\cdot\,,0\}$. Then there exists a $A>0$ such that
	\begin{equation*}
		\IP(\tau_{\text{rec}}>t)\sim A t^{-a} \quad \text{ for all } t>0.
	\end{equation*}
	This was shown in \cite{formanov1986markov} and \cite{formanov1987markov} for $a\geq 1$ and in \cite{vatutin1987conditional} and \cite{zubkov1987conditional} for $a\in (0,1]$.
\end{example}

\begin{example}\label{ex:BirthDeathProcessExponetialTails}
	Let $\birth(n)=\alpha n$ and $\death(n)=\beta n$ with $0\leq \alpha<\beta$. Then by standard methods it can be shown that there exist $A,B>0$
	\begin{equation*}
		\IP(\tau_{\text{rec}}>t)\sim A \exp(-Bt) \quad \text{ for all } t>0.
	\end{equation*}
\end{example}
An interesting special case is $\Lambda(n)=\lambda$ for all $n\geq 0$, with $\birth$ and $\death$ chosen as in Example~\ref{ex:BirthDeathProcess}. The projection $(\1_{\{\bfeta_t> 0\}})_{t\geq 0}$ of the resulting process $\bfeta$ can be interpreted as a contact process with power-law distributed recovery times. Note that $\bfeta$ is a Markov process, while $(\1_{\{\bfeta_t> 0\}})_{t\geq 0}$ is generally not. Therefore, a significant advantage of working with the CPVL $\bfeta$ instead of $(\1_{\{\bfeta_t> 0\}})_{t\geq 0}$ or other contact processes with non-exponential recovery times is that we can utilize the rich theory of Markov processes in this case.

First of all we show that the criticality of the survival probability of the CPVL does not depend on the precise initial configuration. 
\begin{proposition}\label{prop:CriticalValueIndependentOfStart}
	It holds that 
	\begin{equation*}
		\IP(|\bfeta^{\zero}_t|>0 \,\forall\, t\geq 0)>0\,\, \Leftrightarrow\,\,\IP(|\bfeta^{\eta}_t|>0 \,\forall\, t\geq 0)>0
	\end{equation*}
	for all $\eta \in\N^V$  with $0<|\eta|<\infty$.
\end{proposition}

Next, we will provide a sufficient criterion that the CPVL has positive probability to survive. This can be proven via a coupling with a standard contact process. 
\begin{proposition}\label{thm:upperbound}
	If $\Lambda(1)> \death(1)\lambda^{CP}_c$, where $\lambda^{CP}_c$ denotes the critical infection rate of a standard contact process on $G$, then the CPVL has a positive survival probability.
\end{proposition}
On the other hand, it is not clear for which functions $\Lambda$, the infection goes extinct almost surely, i.e.\ that the survival probability of the CPVL is $0$. Therefore, we will next state a sufficient condition when the CPVL dies out almost surely. This is especially of interest in the cases where the recovery times of a vertex are heavy-tailed distributed.
\begin{theorem}\label{thm:LowerBound}
	If $\IE\Big[\int_{0}^{\tau_{\text{rec}}}\Lambda(X_s)\mathrm{d}s\Big]\leq D^{-1}$, then the CPVL dies out almost surely.
\end{theorem}

Since $\bfeta$ is a Markov process, a classical tool to study the long time behaviour are invariant distributions. Thus, we establish a sufficient criterion when an upper invariant law $\overline{\nu}$ exists, which concentrates on the state space $\N_0^V$. The upper invariant law $\overline{\nu}$ is the largest invariant law of $\bfeta$; that is if $\nu$ is another invariant law of $\bfeta$, then $\nu \preceq \overline{\nu}$, where $\preceq$ denotes the stochastic order. This criterion depends solely on the recovery time $\tau_{\text{rec}}$.
\begin{theorem}\label{thm:InvariantDistributionCPVL}
	If $\IE[\tau_{\text{rec}}]<\infty$, then there exists an upper invariant law $\overline{\nu}$ for $(\bfeta_t)_{t\geq 0}$. In particular, $\overline{\nu}(\N_0^V)=1$ and $\overline{\nu}(\{\szero\})\in \{0,1\}$.
\end{theorem}

At the end of this section we want to briefly discuss an example that describes a population where individuals become more infectious as their viral load increases. A possible choice to model this situation is to set the infection rate function to $\Lambda(n)=\lambda \cdot n^{\gamma}$ for all $n\geq 0$ with $\lambda,\gamma\geq 0$ and consider $\birth$ and $\death$ chosen as in Example~\ref{ex:BirthDeathProcess} or \ref{ex:BirthDeathProcessExponetialTails}.

In this case, there exists a critical infection parameter for survival $\lambda_c(\gamma)$. Note that in Section~\ref{sec:Construction} we show that the process $\bfeta$ is monotone with respect to $\lambda$, and thus, we can define
\begin{equation*}
	\lambda_c(\gamma):=\sup\{\lambda>0: \IP_{\lambda}(\exists t\geq 0: |\bfeta^{\zero}_t|=0)=1\}.
\end{equation*}
By Proposition~\ref{prop:CriticalValueIndependentOfStart} the critical value does not depend on the initial configuration $\bfeta_0$ as long as it is finite and not $0$ everywhere. In the case of constant infection rates, i.e.\ $\gamma=0$, we write $\lambda_c=\lambda_c(0)$. The following result is a consequence of Proposition~\ref{thm:upperbound} and Theorem~\ref{thm:LowerBound}.
\begin{corollary}\label{Cor:SurvivalForSpecialCase}
	Let $\Lambda(n)=\lambda n^{\gamma}$ with $\gamma,\lambda\geq 0$. 
	\begin{enumerate}
		\item If $\birth(n)=(n+(1-a)_{+})\1_{\{n>0\}}$ and $\death(n)=(n+(a-1)_{+})\1_{\{n>0\}}$ with $a> 1$, then $\lambda_c(\gamma)\in(0,\infty)$ if $\gamma<a-1$.
		\item If $\birth(n)=\alpha n$ and $\death(n)=\beta n$ for all $n\geq 0$ with $\alpha<\beta$, then $\lambda_c(\gamma)\in(0,\infty)$ for all $\gamma\geq 0$.
	\end{enumerate} 
\end{corollary}

It seems plausible that there are choices of functions $\birth$ and $\death$ such that $\tau_{\text{rec}}<\infty$ almost surely, but $\lambda_c(\gamma)=0$. We plan to study the existence of such a trivial phase transition in future work. See Section~\ref{sec:Discussion} for more details.

In the next section we will consider the special case, $\Lambda(n)\equiv\lambda>0$ for all $n\geq 1$ and state a duality relation. Thus, we briefly summarise that if $\IE[\tau_{\text{rec}}]<\infty$, Proposition~\ref{thm:upperbound} and Theorem~\ref{thm:LowerBound} imply
\begin{equation*}
	(D \IE[\tau_{\text{rec}}])^{-1}\leq \lambda_c\leq \death(1)\lambda^{CP}_c,
\end{equation*}
i.e.\ there is non-trivial phase transition. 
\subsection{Contact process with lingering infections}
In this section, we assume that $\supp(b)=\supp(d)=\N$. We call the Markov process $(\bfxi_t)_{t\geq0}$ that takes values in $\N_0^V$ \emph{contact process with lingering infection}, which we abbreviate with CPLI. We call $x$  active at time $t$ if $\bfxi_t(x)=0$ and dormant if $\bfxi_t(x)>0$ where $\bfxi_t(x)\in \N_0$ describes the level of dormancy. If $\bfxi_t$ is in state $\xi$, then we have for every $x$ the transitions
\begin{equation*}
	\begin{aligned} 
		\xi(x)&\to 0\hspace{4.1em} \text{at rate } \lambda\cdot |\{y\in \cN_x: \xi(y)=0\}|,\\
		\xi(x)&\to \xi(x)+1\quad \text{ at rate } \death(\xi(x)+1)
		\text{ and }\\
		\xi(x)&\to \xi(x)-1\quad \text{ at rate } \birth(\xi(x)),
	\end{aligned}
\end{equation*}
where $\lambda>0$. We again construct the CPLI $\bfxi=(\bfxi_t)_{t\geq 0}$ via a Poisson construction, such that all $\bfxi^{\xi}=(\bfxi^{\xi}_t)_{t\geq 0}$ with $\bfxi_{0}^{\xi}=\xi$ are defined on the same probability space.

In fact there exists a Siegmund duality relation of this process with the contact process with viral load with a constant infection rate.
\begin{proposition}\label{prop:Duality}
	Let $\Lambda(n)=\lambda\in(0,\infty)$ for all $n\geq 0$. The two models $(\bfxi^{\xi}_t)_{t\geq 0}$ and $(\bfeta^{\eta}_t)_{t\geq 0}$ satisfy a Siegmund duality, i.e.\ for $\xi,\eta\in \N_0^V$ it holds that
	\begin{equation*}
		\IP(\xi\geq \bfeta^{\eta}_t)=\IP(\bfxi^{\xi}_t\geq \eta) \qquad \text{ for all } t\geq 0.
	\end{equation*}
	In particular, this implies that $\lim_{t\to \infty}\IP(\bfxi^{\szero}_t(x)>0)$ exists and 
	\begin{equation*}
		\IP(\exists t\geq 0: |\bfeta^{x}_t|=0)=\lim_{t\to \infty}\IP(\bfxi^{\szero}_t(x)>0).
	\end{equation*}
\end{proposition}
This model is again monotone with respect to the initial configuration, i.e.\ if $\xi\leq \xi'$, then $\bfxi^{\xi}_t\leq \bfxi^{\xi'}_t$ for all $t\geq 0$. Moreover, the function $\lambda\mapsto \lim_{t\to \infty} \IP_{\lambda}(\bfxi_t^{\szero}(x)>0)$ is non-increasing in $\lambda$. Therefore, with the duality relation in Proposition~\ref{prop:Duality} it follows that   
\begin{equation*}
	\lambda_{c}=\inf\{\lambda>0: \lim_{t\to \infty}\IP_{\lambda}(\bfxi_t^{\szero}(x)>0)<1\},
\end{equation*}
i.e.\ the critical value $\lambda_c$ could also have a meaning for this model. 
To identify this meaning is our next goal.


The space $\overline{\N}_0^V$ equipped with the product topology is compact, and thus it follows by standard arguments that for every $\lambda>0$ there exists a measure $\overline{\mu}_{\lambda}$ such that
\begin{equation*}
	\bfxi_t^{\szero}\Rightarrow \overline{\mu}_{\lambda}     \quad \text{ as } t\to \infty.
\end{equation*}
We will now see that $\lambda_c$ in fact indicates the phase transition that for $\lambda>\lambda_c$ all the mass of $\overline{\mu}_{\lambda}$ concentrates on $\N_0^V$ versus for $\lambda<\lambda_c$ the limit measure $\overline{\mu}_{\lambda}$ totally degenerates in the sense that the whole mass is on the all $\infty$ configuration.
\begin{theorem}\label{thm:PhasetrastionCPLI}
	If $\lambda>\lambda_c$, then $\overline{\mu}_{\lambda}(\N_0^V)=1$. On the other hand, if $\lambda<\lambda_c$, then $\overline{\mu}_{\lambda}(\xi: \exists y\in V \text{ s.t. } \xi(y)<\infty)=0$.
\end{theorem}

\begin{remark}
	One could also consider the case $|\supp(b)|<\infty$, that is, that there is a maximal capacity $K\in \N$ for the dormancy state. In this case, it is easy to see that the configuration $\xi_K(x)=K$ for all $x\in V$ is an absorbing state for the process $(\bfxi_t)_{t\geq 0}$. This would be a model for infections with a delayed recovery, where there is a chance of a relapse before an individual is fully healed.
	\begin{enumerate}
		\item Here, it becomes apparent that the duality relation in Proposition~\ref{prop:Duality} is a generalisation of the self-duality of the classical contact process. This can be seen by choosing $K=0$, i.e.\ $\death(1)=1$, $\birth(n)=0$ and $\death(n+1)=0$ for all $n\geq 1$, since for this choice $\bfeta$ as well as $(1-\bfxi_t)_{t\geq 0}$ are contact processes.
		\item A similar result to Theorem~\ref{thm:PhasetrastionCPLI} holds, which states that if $\lambda>\lambda_c$, then $\overline{\mu}(\xi: \exists y\in V \text{ s.t. } \xi(y)<K)>0$, and if $\lambda<\lambda_c$, then $\overline{\mu}(\xi: \xi=\xi_K )=1$.
	\end{enumerate}
\end{remark}

\textbf{Outline of the article.} The remainder of the article is structured as follows. In Section~\ref{sec:Discussion}, we put our models in the context of existing literature and point out connections, differences, and some further research directions. In Section~\ref{sec:Construction}, we provide an explicit Poisson construction for both process and use it to show some basic properties such as monotonicity of the systems. This is followed by Section~\ref{sec:SurvivalCPVL}, where we show Proposition~\ref{prop:CriticalValueIndependentOfStart} to Theorem~\ref{thm:InvariantDistributionCPVL} and also the application to our example Corollary~\ref{Cor:SurvivalForSpecialCase}. Section~\ref{Sec:PhaseTransitionCPLI} is devoted to the proof of Theorem~\ref{thm:PhasetrastionCPLI}.
\section{Discussion}\label{sec:Discussion}
We have introduced two multi-type contact processes, where the first process, the CPVL, models a population of individuals, each of which has an internal viral load that determines the infectivity and the ability to recover of an individual. On the other hand, CPLI describes a population where all individuals are infected and can never truly recover, but there are only occasional outbreaks of the infection. 

There is a large body of literature of various variants of multi-type contact processes with state-dependent rates, which study the impact on the survival probability or other aspects of the process. This research direction has been ongoing for quite some time and is to this day still fairly active. See, for example, \cite{krone1999stage}, \cite{blath2023switching} and \cite{belhadji2024aysmptomatic}. One variant which bears resemblance to our processes, is the contact process with ageing introduced by Deshayes~\cite{deshayes2014contact}. In that paper the author studied the asymptotic limiting shape of the infection region. In this variant, when a healthy individual is infected, it is assigned an age of $1$, which increases by one at a constant rate. Then, individuals with age $i$ infect their neighbours with rate $\lambda_i$ and they recover with rate $1$, where $(\lambda_i)_{i\geq 0}$ is assumed to be uniformly bounded.  Since the $\lambda_i$ are uniformly bounded, unlike in our models, both recovery and infection times will always have exponential tails. It would be interesting to compare the CPVL with a contact process with ageing, where infections and recoveries are age-dependent with respect to unbounded sequences $(\lambda_i)_{i\geq 0}$ and $(r_i)_{i\geq 0}$.

As we highlighted previously, the CPVL allows for modelling recovery times whose distribution includes not only exponential tails but also power-law distributions with an exponent $a$. One approach to incorporate non-exponential recovery times is by introducing a suitable static random environment affecting the recovery rates. See, for example, Bramson, Durrett and Schonmann~\cite{Bramson1991enviroment} or Newman and Volchan~\cite{Newman1996persistent}. 

This approach results in a spatially heterogeneous situation, where long recovery times of individuals are an inherent trait given through the random environment. A model choices that leads to a homogenous situation, would be a generalised contact process, in which we assign a recovery time $\tau_{\text{rec}}$ to an individual upon infection, distributed according to some distribution $\mu$. Note that if $\mu$ is an exponential distribution, this yields the classic contact process. As the attentive reader might have noticed, if we neglect the state dependency of the infection rate, i.e.\ set $\Lambda\equiv\lambda$, then the projection of the CPVL $(\1_{\{\bfeta_t> 0\}})_{t\geq 0}$ becomes such a generalised contact process, where $\tau_{\text{rec}}$ is chosen as the extinction time of a BD-process. Hence, one way of interpreting the CPVL is that it is a Markovianisation of a subclass of these generalised contact processes, which we obtain by extending the state space. A similar idea was also used in the context of population genetics models with dormancy in the work of Grevens, den Hollander and Ooman \cite{greven2022spatial}, who constructed a model for spatial populations with multi-layered seed banks.

Another way to incorporate this was introduced by Fontes et al.~\cite{fontes2019renewalI}. They introduced a so-called renewal contact process, which is constructed through a generalisation of the graphical representation used to define the classical contact process. In this framework, the recovery events are modelled by a family of independent renewal processes, with the time increments distributed according to some law $\mu$ instead of Poisson processes.
In a follow-up article \cite{fontes2023renewal} the critical infection rate for survival $\lambda_c(a)$ was studied in the case $\mu[t,\infty) \sim t^{-a}$, i.e.\ $\mu$ being a power-law distribution with exponent $a$. It was shown that $\lambda_c(a)>0$ if $a>1$ and $\lambda_c(a)=0$ if $a<1$. Remarkably, for $a<1$, the renewal contact process has a positive probability of survival even on finite graphs (see \cite{fontes2021finite}).

Obviously, a major difference of the CPVL compared to the renewal contact process is that the latter is, in general, no longer a Markov process. Thus, to further analyse the CPVL one can use the rich theory of Markov processes to gain more insights. However, beyond that, there are some more significant differences. Let $\birth,\death$ be chosen as in Example~\ref{ex:BirthDeathProcess} (so that $\tau_{\text{rec}}$ is distributed according to a power law with exponent $a>0$). For the CPVL, the recovery times are directly coupled to the infection times, such that the time until recovery follows a power-law distribution with exponent $a$. However, for the renewal contact process, the distribution of the increments between consecutive potential recoveries is given by the renewal process. This does not necessarily lead to the same result as if an individual recovers after a time period distributed according to $\mu$ following an infection. The distinction is especially clear for $a<1$, since in \cite[Proposition~7]{fontes2019renewalI} it is shown that there exists some $\varepsilon>0$ so that for large $n$ time intervals of length up to $2^{\varepsilon n}$ contain one or more recovery events with a probability lower than $2^{-\varepsilon n}$.

This is exploited in the proof strategy to show that $\lambda_c(a)=0$ in \cite{fontes2019renewalI}.  One could say that the typical strategy is used to show absence of a subcritical regime, which is to find a finite substructure that can keep the infection alive for an exponential length of time. This provides the infection with enough time to find the next substructure in a different part of the graph. Another example in which this is used is for the standard contact process on Galton-Watson trees with a heavy-tailed offspring distribution. Here a so-called star, a vertex with an exceptional high number of offspring, takes on this role; see Huang and Durrett \cite{huang2020contact} for more details.

For the CPVL, this is not necessarily the case. If we again consider $\birth,\death$ chosen as in Example~\ref{ex:BirthDeathProcess} with $a\in(0,1)$ and assume that an individual has viral load $n$, then the recovery time is typically only of polynomial order in $n$, which is significantly less. Thus, it is not clear whether $\IE[\tau_{\text{rec}}]=\infty$ implies that $\IP(|\bfeta^{\zero}_t|>0 \,\forall\, t\geq 0)>0$ for any $\Lambda\not\equiv 0$. However, it seems plausible that recovery times are typically long enough that this holds true.

Since the Markov structure is still present, we can use it to study the long-time behaviour of the CPVL which is described by the invariant laws of the system. We would like to further study the properties of these laws, for example, for $a\in(0,1)$ it is not clear if a non-trivial upper invariant law concentrating on $\N_0^V$ exists or not. In general, it would be desirable to study the connection between survival probability and the upper invariant law. Since there is no self-duality relation, it is not clear how to do this. Beyond the framework, considered in this article, one could model the viral load with a more general branching structure, which allows not only binary branch events, or consider a viral load with values in $\IR_+$, which evolves according to a diffusion process that is absorbing in $0$. 

The CPLI is interesting not only because of the dual relation to the CPVL, but also on its own. For example, it can also be interpreted as a reactive dormancy model in a spatially structured population, where dormant individuals turn active again on their own or due to an active neighbour which activates them. Here again $\birth,\death$ chosen as in Example~\ref{ex:BirthDeathProcess} are of particular interest. If there is no interaction between neighbours, then it is clear that time an individual stays dormant follows a power-law distribution. Is this preserved if an interaction is introduced? How does the distribution of the reactivation times of dormant particles depend on $\lambda$?
\section{Construction of the processes}\label{sec:Construction}
In this section we present an explicit Poisson constructions of both processes the CPVL and the CPLI. This construction is a related to the random mapping representation in \cite{swart2022course}. However, there are some significant differences, since we consider a non-compact state space and the rate functions of our processes are not bounded.

We introduce the notion of oriented edges. For every $x,y \in V$ with $\{x,y\}$ we denote by $(x,y)$ the oriented edge that points from $x$ to $y$ and vice versa $(y,x)$ points from $y$ to $x$. We denote by $\Vec{E}:=\{(x,y):\{x,y\}\in E\}$ the set of all oriented edges.

Let $\II$ be a standard Poisson point process (PPP) on $\Vec{E}\times [0,\infty) \times \IR$, i.e.\ the intensity measure is the Lebesgue measure $dt$, and $\IU$ and $\ID$ are two independent standard PPP on $V\times [0,\infty)\times \IR$, which are independent of $\II$.

For technical reasons, we first need to construct the processes on a finite truncation. Thus, before we start, let $(V_N)_{N\in \N}$ be an increasing sequence of finite subsets of $V$, which approximates $V$. This means that $|V_N|<\infty$ and $V_N\nearrow V$ as $N\to \infty$, i.e.\ $V_N\subset V_{N+1}\subset V$ for all $N\geq 0 $ and $\bigcup_{N=1}^{\infty}V_N=V$. Then, we define the truncated edge set $E_N:=\{\{x,y\}\in E: x,y\in V_N\}$, and  $\Vec{E}_N$ is analogously defined. In the following we consider truncations of previously defined point processes, which we indicate by a superscript $N$, for example $\II^{N}=(\II^{N}_{(x,y)})_{(x,y)\in \Vec{E}_N}$.

Now we will construct for an arbitrary initial configuration $\eta\in \N_0^{V}$ and an arbitrary starting time $s\in\IR$ the truncated process $(\bfX^N_{s,t}[\eta])_{t\in[s,\infty)}$, i.e.\ $\bfX^N_{s,s}[\eta](x)=\eta(x)$ for all $x\in V_N$. Note that $\bfX^N$ takes values in the countable state space $\N^{V_N}$, which allows us to explicitly construct this Markov process via a Poisson construction and its generator is given by
\begin{equation}\label{eq:TruncatedGenerator}
	\begin{aligned}
		\cG^{N} f(\eta)=&\sum_{x\in V_N}\birth(\eta(x))(f(\eta+\delta_x)-f(\eta))+\death(\eta(x))(f(\eta-\delta_x)-f(\eta))\\
		&+\sum_{y:\{x,y\}\in E_N} \Lambda(\eta(y)) (f((\eta\vee\delta_x))-f(\eta).  
	\end{aligned}
\end{equation}
for all bounded functions $f:\N_0^{V_N}\to \IR$. Recall that we assume that the rate function $\Lambda,\birth,\death:\N_0\to \N_0$ satisfy Assumption~\ref{ass:RateAssumption}.

In the following construction we will suppress the superscript $N$ and only mention it if it is of relevance. Now fix some configuration $\eta\in \N_0^V$ as initial configuration at time $s\in \IR$. Then, we define the jump times $(T_k)_{k\geq 0}=(T^N_k)_{k\geq 0}$ as follows. We set $T_0=s$ and $\bfX^N_{s,s}[\eta]=\eta$. Assume that $T^N_{n-1}\geq s$ for some $n\geq 1$ and $(\bfX^N_{u,s}[\eta])_{u\leq T_{n-1}^{N}}$ is already defined, then we set
\begin{equation}\label{eq:PotentialJumpTimes}
	\begin{aligned}
		I_{n}^N(\eta)&:=\inf\{t\geq T_{n-1}: \exists ((y,x),(\alpha,t)) \in \II^N \text{ s.t. } \alpha<\Lambda(\bfX^N_{s,T_{n-1}}[\eta](y)) \},\\
		U^N_{n}(\eta)&:=\inf\{t\geq  T_{n-1}: \exists (x,(\alpha,t)) \in \IU^N \text{ s.t. } \alpha< \birth(\bfX^N_{s, T_{n-1}}[\eta](x)) \},\\
		D^N_{n}(\eta)&:=\inf\{t\geq  T_{n-1}: \exists (x,(\alpha,t)) \in \ID^N \text{ s.t. } \alpha< \death(\bfX^N_{s,T_{n-1}}[\eta](x))\}.
	\end{aligned}
\end{equation}
We set $T^N_{n}=T^N_{n}(\eta):=I^N_{n}(\eta)\wedge U^N_{n}(\eta) \wedge D^N_{n}(\eta)$ and denote by $\alpha_{n}\in [0,\infty)$, $x_{n}\in  V_N$ or $(y_{n},x_{n})\in \Vec{E}$ the associated points. Next, we define the process after time $T_{n-1}$ as follows. First, we set $\bfX^N_{t,s}[\eta]\equiv \bfX^N_{s,T_{n-1}}[\eta]$ for all $t\in (T_{n-1},T_{n})$, then at time $T_{n}$ we set
\begin{align*}
	\bfX^N_{s,T_{n}}(\eta,x_{n})=\begin{cases}
		\bfX^N_{s,T_{n-1}}[\eta](x_{n})\vee 1 &\text{ if }  ((y_{n},x_{n}),(\alpha_{n},T_{n}))\in \II^N \text{ and } \bfX^N_{s,T_{n-1}}[\eta](y_{n})>0
		\\
		\bfX^N_{s,T_{n-1}}[\eta](x_{n})+ 1 &\text{ if }  (x_{n},(\alpha_{n},T_{n}))\in \IU^N\\
		\bfX^N_{s,T_{n-1}}[\eta](x_{n})- 1 &\text{ if } (x_{n},(\alpha_{n},T_{n}))\in \ID^{N}\\
		\bfX^N_{s,T_{n-1}}[\eta](x_{n}) &\text{ otherwise}
	\end{cases}
\end{align*}
and $\bfX^N_{s,T_{n}}[\eta](y)=\bfX^N_{s,T_{n-1}}[\eta](y)$ for all $y\neq x_{n}$. See Figure~\ref{fig:PoissonConstruction} for a visualisation of the construction.
\begin{figure}[t]
	\centering 
	\includegraphics[width=70mm]{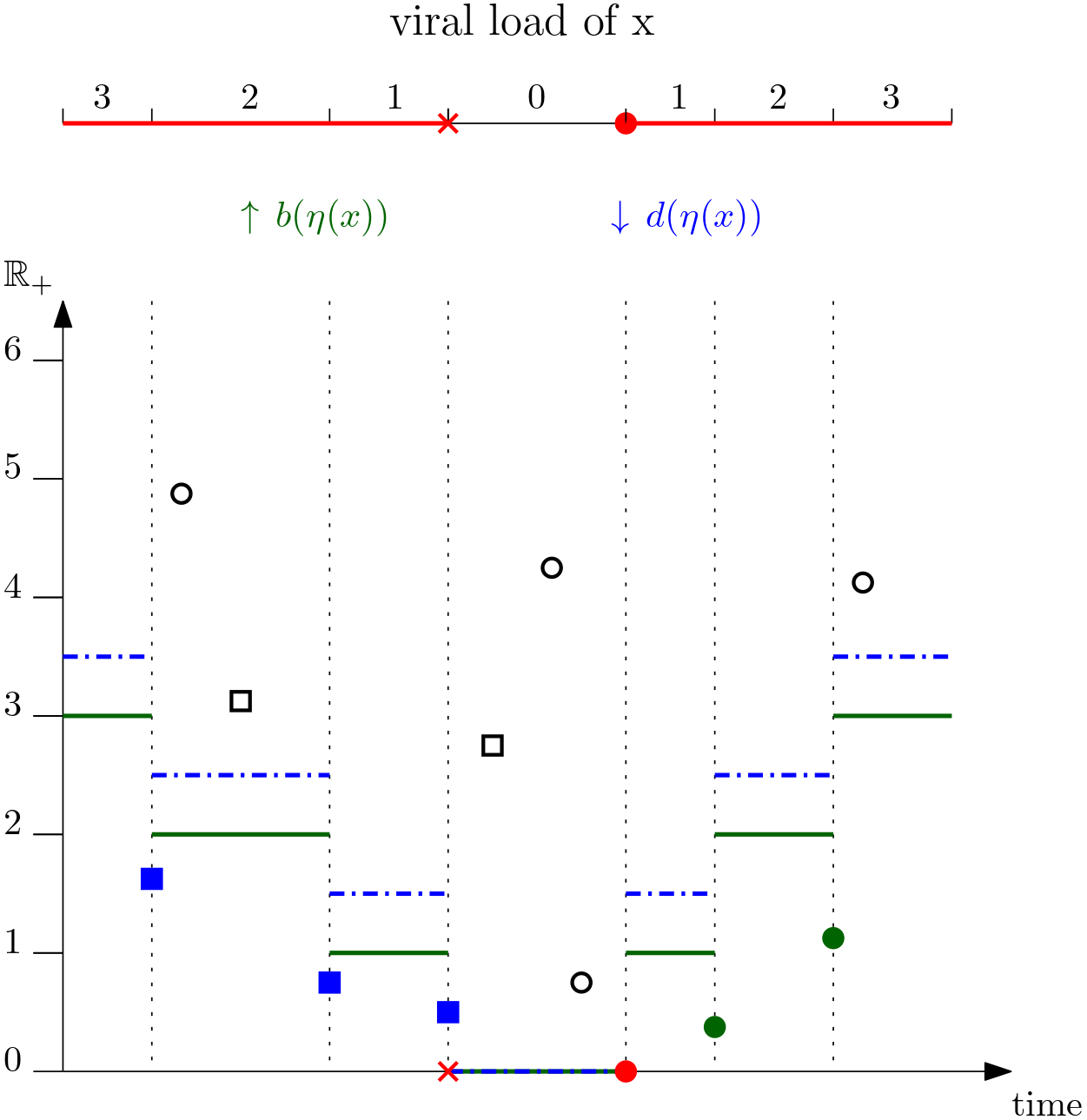}
	\caption{Visualisation of the Poisson construction of the viral load on vertex $x$ with rate function $\birth(n)=n$ and $\death(n)=n+\tfrac{1}{2}$ for all $n\geq 1$. The squares and circles signify points from $\IU$ and $\ID$, where the coloured ones are part of the construction. Furthermore, red cross indicates a recovery of $x$ and the red dot an infection from a neighbour.}
	\label{fig:PoissonConstruction}
\end{figure}
Note that it follows by this construction that $\bfX^N[\eta]\leq \bfX^M[\eta]$ for all $N\leq M$ and that 
\begin{equation}\label{ConcatinationMarkov}
	\bfX^N_{u,t}[\bfX^N_{s,u}[\eta]]=\bfX^N_{s,t}[\eta]
\end{equation}
for any $u\in [s, t]$. Now \eqref{ConcatinationMarkov} directly implies that $(\bfX^N_{s,t}[\eta])_{t\geq s}$ is a Markov process for any $\eta\in \N^V$. This construction yields a well-defined Markov process up to explosion, i.e.\ for all times $t< \inf\{u>s:|\bfX^N_{s,u}[\eta]|<\infty\}$. It only remains to show that this process does not explode in finite time. This can be obtained by a coupling with a system of independent BD-process. To formulate this coupling let us introduce some more notation
\begin{align*}
	\cI^N_{s,t}((y,x),\eta)&:= \{u \in [s,t]: \exists \alpha\in [0,\infty)\text{ s.t. } (\alpha,u) \in \II^N_{(y,x)} \text{ and } \alpha<\Lambda(\bfX^N_{s,u-}[\eta](y)) \},\\
	\cU^N_{s,t}(x,\eta)&:=\{u \in [s,t]: \exists \alpha\in [0,\infty) \text{ s.t. } (\alpha,u) \in \IU^{N}_x \text{ and } \alpha< \birth(\bfX^N_{s,u-}[\eta](x))\},\\
	\cD^N_{s,t}(x,\eta)&:=\{u \in  [s,t]: \exists \alpha\in [0,\infty) \text{ s.t. } (\alpha,u) \in \ID^{N}_x \text{ and } \alpha< \death(\bfX^N_{s,u-}[\eta](x))\},
\end{align*}
where $x,y\in V_N$ and otherwise if $x\notin V_N$ or $(y,x)\notin \vec{E}_N$ we define the corresponding set as the empty set. Next we define
\begin{equation*}
	\begin{aligned}
		\cI^{\vee}_{s,t}((y,x),\eta)&:=\bigcup_{N\geq 1} \cI^N_{s,t}((y,x),\eta),\\   \cU^{\vee}_{s,t}(x,\eta):=\bigcup_{N\geq 1}  \cU^N_{s,t}(x,\eta)\quad &\text{ and }\quad 
		\cD^{\vee}_{s,t}(x,\eta):=\bigcup_{N\geq 1}  \cD^N_{s,t}(x,\eta)
	\end{aligned}
\end{equation*}
for all $x\in V$ and all $(y,x)\in \vec{E}$.

\begin{proposition}\label{prop:CouplingWithIndBD}
	Let $s\in \IR$ and $\eta\in \N_0^V$. There exists a family of i.i.d.\ 
	BD-processes $(\bfZ_{s,t}[\eta])_{t\geq s}$ with $\bfZ_{s,s}(\eta)=\eta\vee \underline{1}$, where $\underline{1}(x)=1$ for all $x \in V$, such that 
	\begin{equation*}
		\sup_{N\geq 1}\bfX^N_{s,t}[\eta]\leq\bfZ_{s,t}[\eta] \quad \text{ for all } \quad t\geq s.
	\end{equation*}
	This means that the process $\bfZ_{s,\cdot}[\eta]$ takes values in $\N^{V}$, 
		and for any $x\in V$, if $\bfZ_{s,t}(\eta)(x)$ is in state $n$ it has transitions 
	\begin{equation*}
		\begin{aligned}
			n\to n+1 &\text{ with rate } \birth(n)\\
			n\to n-1 &\text{ with rate } \death(n)\1_{\{n\geq 2\}}.
		\end{aligned}
	\end{equation*}
\end{proposition}
\begin{proof}
	To shorten the notation we set $\death'(n):=\death(n)\1_{\{n\geq 2\}}$. Define $\bfZ_{s,s}[\eta]:=\eta\vee \underline{1}$ and $T'_0:=s$. Suppose that the process is constructed until $T'_{n-1}$. Then, define
		\begin{equation}\label{eq:PotentialJumpTimesOfZ}
				\begin{aligned}
						U'_{n}(\eta)&:=\inf\{t\geq  T'_{n-1}: \exists (\alpha,t) \in \IU_x \text{ s.t. } \alpha< \birth(\bfZ_{s, T'_{n-1}}[\eta](x)) \},\\
						D'_{n}(\eta)&:=\inf\big\{t\geq  T'_{n-1}: \exists (\alpha,t) \in \ID_x \text{ s.t. } \alpha< \death'(\bfZ_{s, T'_{n-1}}[\eta](x)) \big\}.
					\end{aligned}
			\end{equation}
		Again set $T'_{n}:= U'_{n}(\eta) \wedge D'_{n}(\eta)$ 
		and denote by $(\alpha_{n},x_{n})\in [0,\infty)\times V_N$ or $(y_{n},x_{n})\in \Vec{E}$ the associated point. 
		We again set $\bfZ_{t,s}(\eta)= \bfZ_{s,T'_{n-1}}(\eta)$ for all $t\in (T'_{n-1},T'_{n})$, at time $T'_{n}$ we set
		\begin{align*}
				\bfZ_{s,T'_{n}}(\eta,x_{n})=\begin{cases}
						\bfZ_{s,T'_{n-1}}[\eta](x_{n})+1 &\text{ if }  (x_{n},(\alpha_{n},T'_{n}))\in \IU\\
						\bfZ_{s,T'_{n-1}}[\eta](x_{n})-1 &\text{ if }  (x_{n},(\alpha_{n},T'_{n}))\in \ID\\
						\bfZ_{s,T'_{n-1}}[\eta](x_{n}) &\text{ otherwise}
					\end{cases}
			\end{align*}
		and $\bfZ_{s,T'_{n}}[\eta](y)=\bfZ_{s,T'_{n-1}}[\eta](y)$ for all $y\neq x_{n}$. 
		
		Fix some $N\geq 0$ and recall that $T_n^N$ is the $n$-th potential jump time of $\bfX^N_{s,\cdot}[\eta]$, see \eqref{eq:PotentialJumpTimes}. We denote by
		\begin{equation*}
				S_k\in \{T_n: n\geq 0\}\cup \{T'_n: n\geq 0\}
			\end{equation*}
		all potential jump times of $\bfZ$ and $\bfX^N$.
		
		Now we show that if $\bfX^N_{s,S_{n-1}}[\eta]\leq \bfZ_{s,S_{n-1}}[\eta]$, then $\bfX^N_{s,S_{n}}[\eta]\leq \bfZ_{s,S_{n}}[\eta]$, i.e.\ at any potential jump time $S_n$ the order is preserved. For that, let $(\alpha_{n},x_{n})\in [0,\infty)\times V$ be the unique associated point to $S_n$. If $\bfX^N_{s,S_{n-1}}[\eta](x_n)=\bfZ_{s,S_{n-1}}[\eta](x_n)$, then by definition of the potential jump times \eqref{eq:PotentialJumpTimes} and \eqref{eq:PotentialJumpTimesOfZ} both processes jump up or down by $1$ jointly. On the other hand if $\bfX^N_{s,S_{n-1}}[\eta]<\bfZ_{s,S_{n-1}}[\eta]$, then the jump at time $S_n$ must preserve $\bfX^N_{s,S_{n-1}}[\eta]\leq \bfZ_{s,S_{n-1}}[\eta]$, since the maximal jump height at location $x_n$ is $1$.
		
		Since $N$ was arbitrarily chosen and $\eta\leq \bfZ_{s,s}[\eta]$ by construction, it follows that 
		\begin{equation*}
				\sup_{N\geq 1}\bfX^N_{s,t}[\eta]\leq\bfZ_{s,t}[\eta] \quad \text{ for all } \quad t\geq s.
			\end{equation*}
		Furthermore, it is easy to see that $\bfZ_{s,\cdot}[\eta](x)$ and $\bfZ_{s,\cdot}[\eta](y)$ are independent for all $x\neq y$ and every $\bfZ_{s,t}[\eta](x)$ is a BD-process on $\N$ with the claimed transitions.
\end{proof}
\begin{corollary}\label{cor:NonExplosivAndFiniteJumps}
	Suppose $\eta\in \N_0^V$ and $x\in V$ and $(x,y)\in \vec{E}$. 
	\begin{itemize}
		\item[(i)] For any $N\in \N$ the process $\big(\bfX^N_{s,t}[\eta]\big)_{t\geq s}$ does not explode in finite time, i.e.\ $|\bfX^N_{s,t}[\eta]|<\infty$ for all $t\geq s$ almost surely.
		\item[(ii)] The random sets $\cI^{\vee}_{s,t}((x,y),\eta')$, $\cU^{\vee}_{s,t}(x,\eta')$ and $\cD^{\vee}_{s,t}(x,\eta')$ are almost surely finite for all $t\geq s$ and all $\eta'\in \N_0^V$ with $\eta(x)=\eta'(x)$.
	\end{itemize}
\end{corollary}
\begin{proof}
	First note that $(i)$ is a direct consequence of Proposition~\ref{prop:CouplingWithIndBD}, since it is assumed that $\birth\in O(n)$ as $n\to\infty$. It is well known that BD-processes with such an assumption imposed on the birth rate function $\birth(\cdot)$ do not explode in finite time, see \cite[Theorem~6.8(19)]{grimmett2020probability}.
	
	Next we will show $(ii)$. Let is first consider $\cI_{s,t}((x,y),\eta)$. For that set
	\begin{equation*}
		\overline{\cI}_{s,t}((x,y),\eta):=\{u \in [s,t]: \exists \alpha\in [0,\infty)\text{ s.t. }  (\alpha,u) \in \II^N_{(y,x)} \text{ and } \alpha<\Lambda(\bfZ_{s,u}[\eta](y)) \}.
	\end{equation*}
	Since $\Lambda:\N_0\to \N_0$ is a monotone function Proposition~\ref{prop:CouplingWithIndBD} implies that 
	\begin{equation*}
		\cI^{\vee}_{s,t}((x,y),\eta)\subset\overline{\cI}_{s,t}((x,y),\eta).
	\end{equation*}
	Since $\II$ is a PPP and by assumption on the function $\Lambda(\cdot)$ it follows that
	\begin{equation*}
		\IE\big[|\overline{\cI}_{s,t}((x,y),\eta)|\big]=\int_{0}^{t}\IE[\Lambda(\bfZ_{s,u}[\eta](x))]ds<\infty
	\end{equation*}
	for every $t\geq s$, which implies that $\cI_{s,t}((x,y),\eta)$ is almost surely finite for every $t\geq s$. 
	
	By assumption, there exists an $C>0$ such that $\birth(n),\death(n)\leq Cn$ for all $n\geq 0$. Thus, we can use $\ID\cup\IU$ to construct a pure birth process $Y^x=(Y^x_t)_{t\geq s}$ with the rate function $\birth+\death$ such that
	\begin{equation*}
		\max\{|\cU^{\vee}_{s,t}(x,\eta)|,|\cD^{\vee}_{s,t}(x,\eta)|\}\leq Y^x_t
	\end{equation*}
	for all $t\geq s$. We know that $\IE[Y_t^x]<\infty$ for all $t\geq s$, since $Y^x$ is a pure birth process with linear rates, and thus $|\cU_{s,t}(x,\eta)|$ and $|\cD_{s,t}(x,\eta)|$ are finite almost surely for all $t\geq s$. Note that the various upper bounds all only depend on the initial value $\eta(x)$, which provides the claim.
\end{proof}
Next, we extend the Poisson construction to infinite graphs by showing that the limit 
\begin{equation*}
	\bfX_{s,t}[\eta]=\lim_{N\to \infty}\bfX^N_{s,t}[\eta]
\end{equation*}
exists for every initial configuration $\eta\in \N^V$. Then we set $\bfeta^{\eta}_t:=\bfX_{0,t}(\eta)$ and 
we show that this process $\bfeta$ is a Feller process. 

Before we proceed, we first need to study the waiting time $T$ of the first outgoing infection from a vertex $x$ to one particular neighbour $y$ right after it got infected, i.e.\ its viral load is $1$.
Recall that $X=(X_t)_{t\geq 0}$ is a BD-process with rate functions $\birth, \death$ and initial value $X_0=1$. The distribution of $T$ is given through
\begin{equation}\label{eq:DistributionOfInfectionTime}
	\IP(T> t)=\IE\bigg[\exp\bigg(-\int_{0}^{t}\Lambda(X_s)ds\bigg)\bigg] \qquad \text{for all} \quad t\geq 0.
\end{equation}
\begin{lemma}\label{eq:ControlOverSum}
	Let $T_1,\dots, T_n$ be i.i.d.\ and distributed as in \eqref{eq:DistributionOfInfectionTime}. Then, for any $\kappa>0$ there exists an $C=C(\kappa)>0$ such that $\IP(S_n\leq C n)\leq \exp(-\kappa n)$,
	where $S_n:=\sum_{i=1}^{n}T_i$.
\end{lemma}
\begin{proof}
	Let $c>0$, then It follows
	\begin{equation*}
		\IE[e^{-cT}|X]=\int_{0}^{\infty}\Lambda(X_t)\exp\bigg(-\int_{0}^{t}(\Lambda(X_s)+c)ds\bigg)dt=1-c\IE[T_c|X],
	\end{equation*}
	where $T_c$ is distributed as in \eqref{eq:DistributionOfInfectionTime} just with respect to the function $\Lambda(X_t)+c$. Thus, we have that
	\begin{equation*}
		\IE[e^{-cT}]=1-c\IE[T_c]:=\vartheta_c.
	\end{equation*}
	By comparing $T_c$ with an exponential distribution with parameter $c$ we see that $c\IE[T_c]\leq 1$ for all $c>0$. Since, we assumed that $\IE[\Lambda(X_t)]<\infty$ for all $t\geq 0$ it follows that for every $\delta\in(0,1)$ there exists an $\varepsilon'>0$ such that $\int_{0}^{\varepsilon'}\IE[\Lambda(X_t)]<\delta$. With Jensen's inequality follows that
	\begin{equation*}
		c\IE[T_c]\geq  c\int_{0}^{\infty}\exp\bigg(-\int_{0}^{t}(\IE[\Lambda(X_s)]+c)ds\bigg)dt\geq  \int_{0}^{\varepsilon'}ce^{-ct-\delta}dt=e^{-\delta}(1-e^{-c\varepsilon'}).
	\end{equation*}
	Since $\delta$ does not depend on $c$ we get that $\lim_{c\to\infty} c\IE[T_c]\geq e^{-\delta}$, for every $\delta\in (0,1)$. Finally, we let $\delta\to 0$ which implies that $\lim_{c\to \infty} c\IE[T_c]= 1$. This is equivalent to $\vartheta_c\to 0$ as $c\to \infty$. 
	
	By using Markov's inequality we get the following estimate   
	\begin{equation*}
		e^{-c\theta n}\IP(S_n\leq \theta n)=e^{-c\theta n}\IP(e^{-c\theta n}\leq e^{-cS_n})\leq \IE[e^{-c\sum_{i=1}^nT_i}]=\vartheta_c^n.
	\end{equation*}
	Therefore, we get that
	\begin{equation*}
		\IP(S_n\leq \theta n)\leq \vartheta_c^n e^{c\theta n}=\exp\big(\big(c\theta +\log(\vartheta_c) \big) n \big)
	\end{equation*}
	Now choose $\theta_{c,\varepsilon}=c^{-1}\log\big(K_{\varepsilon,c}\vartheta_c^{-1}\big)$, where $K_{\varepsilon,c}:=\vartheta_c+\varepsilon$ and $\varepsilon\in (0,c\IE[T_c])$. This choice guarantees that $\theta_{\varepsilon,c}>0$ and $\Theta_{c,\varepsilon}:=\log(K^{-1}_{c,\varepsilon})>0$. Thus, we get that
	\begin{equation}\label{eq:ExpInequality}
		\IP(S_n\leq \theta_{\varepsilon,c} n)\leq \exp\big(\log(K_{\varepsilon,c}\vartheta_c^{-1}\vartheta_c)  n \big)=\exp(-\Theta_{c,\varepsilon} n)
	\end{equation}
	Note that, since $c\to \infty$ and $\varepsilon\to 0$ the constant $\Theta_{c,\varepsilon}\to \infty$. Thus, for any $\kappa>0$ we find an appropriate choice of the constants $c$ and $\varepsilon$ such that $\Theta_{c,\varepsilon}>\kappa$, while $\theta_{\varepsilon,c}>0$ is guaranteed.
\end{proof}
With this we can control the influence of vertices $y$ on the viral load $\bfX^N_{s,t}[\eta](x)$ of a vertex $x$ at time $t$, when these vertices $y$ are a large distance away. We set $\B_{n}(x):=\{y\in V: r(x,y)\leq n\}$ for all $x\in V$, i.e\ $\B_{n}(x)$ is the ball of radius $n$ around $x$ with respect to the graph distance  $r(\cdot,\cdot)$. Furthermore, we define $\del B_{n}(x):=\B_{n}(x)\backslash\B_{n-1}(x)$ for all $n\geq 1$ as the boundary of the ball.
\begin{proposition}\label{prop:BoundOnPotentialInfluence}
	Let $x\in V$, $s<t$, $\eta\in \N^V$ and
	\begin{equation*}
		\kappa>\sup_{n\geq 1}n^{-1}\log(|\partial\B_n(x)|):=D_{\max}.
	\end{equation*}
	Choose $N_0\in \N$ large enough so that $\kappa r(V_{N_0},x)\geq t-s$. Then it holds that 
	\begin{align}
		&\IP( \bfX^{N_0}_{s,t}[\eta](x)= \bfX^N_{s,t}[\eta](x)\, \forall \, N\geq N_0)\geq 1-\exp(-(\kappa-D_{\max})r(V_{N_0},x))\label{eq:ExponetialBoundsOnInfluence1},\\
		&\IP( \bfX^{N}_{s,t}[\eta](x)= \bfX^N_{s,t}[\eta_{N_0}](x)\, \forall \, N\geq N_0)\geq 1-\exp(-(\kappa-D_{\max})r(V_{N_0},x)),\label{eq:ExponetialBoundsOnInfluence2}
	\end{align}
	for every $\eta_{N_0}\in\N^V$ with $\eta_{N_0}(x)=\eta(x)$ for all $x\in V_{N_0}$.
\end{proposition}
\begin{proof}
	Let us first fix some arbitrary vertex $x\in V_N$. We note that the state of $x$ at time $t$ with respect to $\bfX^{N_0}$ could only change with a larger $N$ if there is a new influence in form  of an infection path starting from some vertex $y\in V_N\backslash V_{N_0}$. 
	
	Thus, it can only hold that $\bfX^{N_0}_{s,t}[\eta](x)\neq \bfX^N_{s,t}[\eta](x)$ if there exists a path of vertices $y=x_0,\dots,x_{n+1}=x$ with $y\in V_{N_0}\backslash V_{N}$, $x_i\in V_{N_0}$ for all $i\in \{1,\dots, n+1\}$ and times $s\leq t_0\leq \dots \leq t_{n}\leq t$ with $n\geq r(x,y)$ such that $(t_{i},(x_{i},x_{i+1}))\in \cI$ and $\bfX^{N}_{s,t_0}[\eta](y)\geq 1$, $\bfX^{N_0}_{s,t_i-}[\eta](x_{i+1})=0$  for all $i\leq n$. We denote such a path by $(s,y)\rightsquigarrow(t,x)$.
	
	Let $X=(X_t)_{t\geq 0}$ be the BD process with rates functions $\birth(\,\cdot\,)$, $\death(\,\cdot\,)$ and $X_0=1$. Set $T_i:=t_{i+1}-t_{i}$ for $0\leq i\leq n$. Since we know that $\bfX_{s,t_i-}[\eta](x_i)=0$ we know that the sequences $(T_i)_i$ is i.i.d.\ and 
	\begin{equation*}
		\IP(T_i>s)=\IE\bigg[\exp\bigg(-\int_{0}^{s}\Lambda(X_u)du\bigg)\bigg].
	\end{equation*}
	In order to apply Lemma~\ref{eq:ControlOverSum} we excluded $t_0-s$, since this increment clearly has a different distribution from the $T_i$. However, because of monotonicity of the measure neglecting it is no problem. If $(s,y)\rightsquigarrow(t,x)$ then this surely implies that $\sum_{i=1}^{n}T_i<t-s$. Let $N$ be large enough so that $t-s\leq \kappa r(V_{N_0},x)$. Thus, we get
	\begin{equation*}
		\IP(S_n\leq t-s)\leq \IP(S_n\leq \kappa r(V_{N_0},x)) \leq \exp(-\kappa r(V_{N_0},x)),
	\end{equation*}
	where we used that $t-s\leq \kappa r(V_{N_0},x)$. Let $n_0:=r(V_{N_0},x)$, then it follows that
	\begin{equation*}
		\sum_{y\in \partial\B_{n_0}(x)}\IP\big((s,y)\rightsquigarrow(t,x)\big)\leq |\partial B_{n_0}(x)|\exp(-\kappa r(V_{N_0},x)) \leq \exp(-(\kappa-D_{\max}) r(V_{N_0},x)).
	\end{equation*}
	Now the first claim follows immediately, i.e.\ that
	\begin{equation*}
		\IP(\exists \, N\geq N_0 : \bfX^{N_0}_{s,t}[\eta](x)\neq \bfX^N_{s,t}[\eta'_{N_0}](x))\leq \exp(-(\kappa-D_{\max})r(V_{N_0},x)).
	\end{equation*}
	For the second claim we use that existence of a path $(s,y)\rightsquigarrow(t,x)$ does not depend in any way on the values of $\bfX^N_{s,t}[\eta](y)$ for $y\notin V_{N_0}$. Thus, we can change the initial configuration of $\bfX_{s,t}^N[\cdot]$ to any other $\eta'_{N_0}$ and nothing changes.
\end{proof}

\begin{theorem}\label{thm:ExistenceAndFeller}
	Let $s\in\IR$ and $\eta\in \N_0^V$. Then for every $t>s$ the limit
	\begin{equation}\label{eq:LimitDefinition}
		\bfX_{s,t}[\eta]=\lim_{N\to \infty}\bfX^N_{s,t}[\eta]
	\end{equation}
	exists.
	Let $\bfeta_0$ be a $\N_0^V$-valued random variable independent of all Poisson processes used in the construction of $\bfX$. Set $\bfeta_t:=\bfX_{t,0}^N(\bfeta_0)$, then $\bfeta=(\bfeta_t)_{t\geq 0}$ is a Feller process with semigroup 
	\begin{equation*}
		P_t(\cdot,\eta)=\IP(\bfX_{0,t}(\eta)\in \,\cdot\,), \quad \text{ where }\quad t\geq 0.
	\end{equation*}
\end{theorem}
\begin{proof}
	Since we consider the product topology, it suffices to show that
	\begin{equation*}
		\bfX_{t,s}[\eta](x)=\lim_{N\to \infty}\bfX^N_{t,s}[\eta](x)
	\end{equation*}
	almost surely for all $x\in V$, in order to show that the limit exists. A direct consequence of Proposition~\ref{prop:BoundOnPotentialInfluence} is the convergences in probability, i.e.\ $\bfX^N_{t,s}[\eta](x)\stackrel{\IP}{\to} \bfX_{t,s}[\eta](x)$. Since the bound in \eqref{eq:ExponetialBoundsOnInfluence1} is summable with respect to $N_0$, the almost sure convergences follows by an application of the Borel-Cantelli-Lemma.
	
	Next let $u\in[s,t]$, then by construction we know that $\bfX^N_{u,t}[\bfX^N_{s,u}[\eta]]=\bfX^N_{s,t}[\eta]$ for any $N\geq 0$. By \eqref{eq:ExponetialBoundsOnInfluence1} there exists a $N_0>0$ and a constant $c>0$ such that 
	\begin{equation*}
		\IP\big(\bfX_{u,t}[\zeta](x)=\bfX^N_{u,t}[\zeta](x)\big|\bfX^N_{s,u}[\eta]=\zeta\big)\geq 1-e^{-cN_0}
	\end{equation*}
	for all $N\geq N_0$ and all $\zeta\in \N^V$, where $c$ does not depend on $\zeta$.
	Thus, we get that
	\begin{equation*}
		\IP\big(\bfX_{u,t}[\bfX^N_{s,u}[\eta]](x)=\bfX^N_{u,t}[\bfX^N_{s,u}[\eta]](x)\big)\geq 1-e^{-cN_0}.
	\end{equation*}
	On the other hand, we find an even larger $N_1>N_0$ such that by \eqref{eq:ExponetialBoundsOnInfluence1} there exists a constant $c'>0$ such that 
	\begin{equation*}
		\IP(\bfX_{s,u}[\eta](y)=\bfX^N_{s,u}[\eta](y) \,\forall\, y\in V_{N_0})>1-e^{-c'N_0} 
	\end{equation*}
	for all $N>N_1$. Finally we can conclude with \eqref{eq:ExponetialBoundsOnInfluence2} that there exists a $N_2>N_1$ and constant $C>0$ such that 
	\begin{equation*}
		\IP\big(\bfX_{u,t}[\bfX_{s,u}[\eta]](x)=\bfX^N_{u,t}[\bfX^N_{s,u}[\eta]](x)\big)\geq 1-e^{-CN_0}
	\end{equation*}
	for all $N\geq N_2$. Thus, for any $x\in V$ we have $\bfX_{u,t}[\bfX_{s,u}[\eta]](x)=\bfX_{s,t}[\eta](x)$ almost surely, which follows by letting $N_0\to \infty$, and thus $N\to \infty$. This shows that the process $\bfeta=(\bfeta_t)_{t\geq 0}$ is indeed a Markov process.
	
	It remains to show that $t\to P_t(x,\cdot)$ is a continuous map from $\N^{V}\times [0,\infty)$ to $\cM_1(\N^{V})$, i.e.\ we need to show that $P_{t_n}(\eta_n,\cdot)\Rightarrow P_{t}(x,\cdot)$ as $(\eta_n,t_n)\to (x,t)$. Since we consider the product topology it suffices to show
	\begin{equation*}
		\bfX_{0,t_n}[\eta_n](x)\to \bfX_{0,t}[\eta](x) \quad \text{ almost surely as } n\to \infty
	\end{equation*}
	almost surely. First, let $T\geq t$ and define
	\begin{equation*}
		\cR_x(s,T,\eta):=\bigcup_{y\in \cN_x}\cI_{(x,y)}(s,T,\eta)\cup \cU_{x}(s,T,\eta)\cup \cD_{x}(s,T,\eta).
	\end{equation*}
	This random set contains all potential jump times that affect $x$ directly between times $s$ and $T$ and it is almost surely finite by Corollary~\ref{cor:NonExplosivAndFiniteJumps}. Thus, for every realisation of the PPP's $\Delta=(\II,\IU,\ID)$ there exists an $\varepsilon=\varepsilon(\Delta, \eta|_{\B_1(x)})>0$ such that $\cR_x(t-\varepsilon,t+\varepsilon,\eta)$ is empty. Note that $\varepsilon$ only depends on the values of $\eta(y)$ for $y\in \B_1(x)$. 
	
	Now we choose $n_0\geq 0$ large enough such that $\eta_n(y)=\eta(y)$ for all $y\in \B_1(x)$ and for all $n\geq n_0$. Then we choose $n_1\geq n_0$ so that $t_n\in (t-\varepsilon,t+\varepsilon)$ for all $n\geq n_1$. Finally, by \eqref{eq:ExponetialBoundsOnInfluence2} we see that for the given realisations of the PPP $\Delta$ there must exist a $N_0>0$ such that if $\eta(y)=\eta'(y)$ for all $x\in V_{N_0}$, then $\bfX_{0,t-\varepsilon}[\eta](x)=\bfX_{0,t-\varepsilon}[\eta'](x)$. Thus, we choose $n_2\geq n_1$ large enough so that $\eta_n(y)=\eta(y)$ for all $y\in V_{N_0}$. But together with the fact that $\cR_x(t-\varepsilon,t+\varepsilon,\eta)$ is empty this implies that 
	\begin{equation*}
		\bfX_{0,t_n}[\eta_n](x)=\bfX_{0,t}[\eta](x) \quad \text{ for all } n\geq n_2.\qedhere
	\end{equation*}
\end{proof}
After we showed that $\bfeta_t:=\bfX_{t,0}^N(\bfeta_0)$ is indeed a Feller process, what remains is to demonstrate its connection to the operator $\cG$ associated with the transitions \eqref{CPViral}, which is given by
	\begin{align*}
		\cG f(\eta)=&\sum_{x\in V}\birth(\eta(x))(f(\eta+\delta_x)-f(\eta))+\death(\eta(x))(f(\eta-\delta_x)-f(\eta))\\
		&+ \sum_{y:\{x,y\}\in E} \Lambda(\eta(y)) (f(\eta\vee \delta_x)-f(\eta)).
	\end{align*}
	Recall that $\cC_{\text{fin}}(\N_0^V)$ is the set of bounded continuous functions which only depends on finitely many coordinates.
\begin{proposition}\label{prop:FiniteCoordinatesDomain}
		Let the process $\bfeta$ be as in Theorem~\ref{thm:ExistenceAndFeller} and $(P_t)_{t\geq0}$ the associated semigroup. It holds that $\lim_{t\to \infty}\big|\big|\tfrac{1}{t}(P_tf-f)\big|\big|_{\infty}=\cG f$ for all $f\in \cC_{\text{fin}}(\N_0^V)$.
	\end{proposition}
	\begin{proof}
		By definition for every $f\in \cC_{\text{fin}}(\N_0^V)$ there exists a set $A\subset V$ such that $|A|<\infty$ and $f((\eta(x))_{x\in V})=f((\eta(x))_{x\in A})$ for all $\eta\in \N_0^V$. Now let $N>0$ be large enough such that 
		\begin{equation}
			\{y\in V: r(A,y)\leq 3 \}\subset V^N.
		\end{equation}
		For this choices of $N$ it holds that $\cG f=\cG^N f$, where $\cG^N$ is defined in $\eqref{eq:TruncatedGenerator}$. Since $\bfX^N_{0,t}[\eta]$ is defined on a countable state space we know that its a Feller process and that $\cG^N$ is the associated generator, i.e.\
		\begin{equation*}
			\lim_{t\to 0} \frac{1}{t}\big|\big|\IE[f(\bfX^N_{0,t}[\cdot])]-f\big|\big|_{\infty}= \cG^Nf.
		\end{equation*}
		Since we already know that $\bfX_{0,t}[\cdot]$ is a Feller process and $\cG f=\cG^N f$ it suffices to show that
		\begin{equation}\label{eq:SufficientConvergences}
			\lim_{t\to 0}\frac{1}{t}\big|\big|\IE[f(\bfX_{0,t}[\cdot])]-|\IE[f(\bfX^N_{0,t}[\cdot])]\big|\big|_{\infty}=0.
		\end{equation}
		To see this let $\eta\in \N^V_0$ and note that
		\begin{align*}
			|\IE[f(\bfX_{0,t}[\eta])]-|\IE[f(\bfX^N_{0,t}[\eta])|&\leq ||f||_{\infty}\IP\big( \exists x\in A: \bfX_{s,t}[\eta](x)\neq \bfX^N_{s,t}[\eta](x)  \big)\\
			\leq |A||\B_{3}(0)|\IP(S_2\leq t).
		\end{align*}
		The inequality in the last line comes from estimating the probability that the state of any $x\in A$ can be influenced by a neighbour in distance $3$. It can be derived by the same line of arguments as in the proof of Proposition~\ref{prop:BoundOnPotentialInfluence}, where we again ignore the first infection time, since its distribution is not the same as in \eqref{eq:DistributionOfInfectionTime}. To conclude \eqref{eq:SufficientConvergences} it remains to show that $\frac{1}{t}\IP(S_2\leq t)\to 0$ as $t\to 0$. 
		
		Recall that $S_2:= T_1+T_2$ with $T_1$ and $T_2$ are independent and identically distributed as specified in \eqref{eq:DistributionOfInfectionTime}, i.e.
		\begin{equation*}
			\IP(T> t)=\IE\bigg[\exp\bigg(-\int_{0}^{t}\Lambda(X_s)ds\bigg)\bigg] \qquad \text{for all} \quad t\geq 0,
		\end{equation*}	
		where  $X=(X_t)_{t\geq 0}$ is a BD-process with rate functions $\birth, \death$ and initial value $X_0=1$. To get a better control of $T$, we estimate the probability that $\sup_{s\leq t}X_s> C\in \N$ for $t$ small and $C\geq2$. Let $\overline{T}_1\stackrel{d}{=}\Gamma(C,b(C+1))$. Then, by a comparison with pure birth process with the rate function $b$, it follows that
		\begin{equation*}
			\IP\bigg(\sup_{s\leq t}|X_s|\geq C\bigg)<\IP(\overline{T}_1< t).
		\end{equation*}
		In the next step, we partition the event space according to whether the BD-process remains uniformly bounded by $C$ until time $t$. Let $\overline{T}_2\stackrel{d}{=}\Gamma(2, \Lambda(C))$. Then, it holds that
		\begin{align*}
			\tfrac{1}{t}\IP(S_{2}<t)\leq \tfrac{1}{t}\big(\IP(\overline{T}_2<t ) +2\IP(\overline{T}_1<t)\big).
		\end{align*}
		By using the explicit form of the Gamma distribution and that $C\geq2$ it follows that the expression convergences to $0$ as $t\to 0$. Thus, we can conclude \eqref{eq:SufficientConvergences}.
	\end{proof}
	
	\begin{proof}[Proof of Theorem~\ref{thm:ExistenceOfCPVL}]
		This is a direct consequence of Theorem~\ref{thm:ExistenceAndFeller} and Proposition~\ref{prop:FiniteCoordinatesDomain}.
	\end{proof}

\subsection{Some consequences of the Poisson construction for the CPVL}
From here onward, we assume that the CPVL $\bfeta$ is constructed via the above described Poisson construction, i.e.\ $\bfeta^{\eta}_t=\bfX_{0,t}[\eta]$ for all $t\geq 0$ and $\eta\in \N_0^V$, with $\bfX$ as in Theorem~\ref{thm:ExistenceAndFeller} and that the associated rate functions $\Lambda,\birth,\death$ always satisfy Assumption~\ref{ass:RateAssumption}. Moreover, it will be convenient for the following proofs to extend the notation for the different types of jump time to the CPVL $\bfeta$. Therefore, we set
\begin{equation}\label{eq:JumpTimesOfCPVL}
	\begin{aligned}
		\cI((y,x),\eta)&:= \{u \in [0,\infty): \exists \alpha\in [0,\infty)\text{ s.t. }  (\alpha,u) \in \II_{(y,x)} \text{ and } \alpha<\Lambda(\bfeta_t^{\eta}(y)) \},\\
		\cU(x,\eta)&:=\{u \in [0,\infty): \exists \alpha\in [0,\infty) \text{ s.t. } (\alpha,u) \in \IU_x \text{ and } \alpha< \birth(\bfeta_t^{\eta}(x))\},\\
		\cD(x,\eta)&:=\{u \in  [0,\infty): \exists \alpha\in [0,\infty) \text{ s.t. } (\alpha,u) \in \ID_x \text{ and } \alpha< \death(\bfeta_t^{\eta}(x))\}.
	\end{aligned}
\end{equation}
\begin{lemma}\label{lem:MonotonictyInRates}
	Let $\bfeta$ be CPVL with rate functions $\Lambda,b$ and $d$. Furthermore, suppose $\Lambda',\birth',\death':\N_0\to \N_0$ satisfy Assumption~\ref{ass:RateAssumption}. Furthermore, suppose that 
	\begin{equation}\label{eq:OrderOfRateFunctions}
		\Lambda'(n)\geq \Lambda(n), \birth'(n)\geq \birth(n) \text{ and } \death'(n)\leq \death(n) \text{ for all } n\geq 0.
	\end{equation}
	then there exists a CPVL $\bfeta'$ with infection rate $\Lambda'$ and rate functions $b'$ and $d'$ such that if $\bfeta_{0}\leq  \bfeta'_{0}$, then $\bfeta_{t}\leq  \bfeta'_{t}$ for all $t\geq 0$.
\end{lemma}
\begin{proof}
	Let us fix an initial configuration $\bfeta_0=\eta\in \N_0^V$. Recall that $\II$, $\IU$ and $\ID$ are the three independent PPP's used in the construction of $\bfeta$. Let $\bfeta'$ be a CPVL with $\bfeta'_0=\eta$, which is constructed using the same PPP's just with respect to the rate functions $\Lambda,'\birth'$ and $\death'$ instead. Also let $\cI'((y,x),\eta)$, $\cU'(x,\eta)$ and $\cD'(x,\eta)$ denote the sets of the various jump times. Suppose $S$ is a jump time of either $\bfeta$ or $\bfeta'$ and let $x\in V$ be the unique associated vertex at which the potential change at time $S$ happens. 
	
	Suppose $\bfeta_{S-}\leq \bfeta'_{S-}$ holds; then it suffices to show that $\bfeta_{S}\leq \bfeta'_{S}$, i.e.\ that every possible jump preserves the order. If $\bfeta_{S-}(x)< \bfeta'_{S-}(x)$, then there is nothing to show, since every jump can at most change the viral load of $x$ by $\pm 1$. Thus, we can restrict ourselves to the case $\bfeta_{S-}(x)= \bfeta'_{S-}(x)$. In this case \eqref{eq:OrderOfRateFunctions} implies that if $S\in \cI((y,x),\eta)$ or $S\in \cU(x,\eta)$, then $S\in \cI'((y,x),\eta)$ or $S\in \cU'(x,\eta)$, that is if $\bfeta_{S-}(x)$ jumps up both process do so together, which preserves the order. On the other hand $S\in \cD'(x,\eta)$ implies $S\in \cD(x,\eta)$, which also preserves the order.
\end{proof}
Next, we show that the model is additive with respect to the maximum operation $\vee$. Note that the model is not necessarily additive in the classical sense, that is with respect to the $+$ operation.
\begin{lemma}\label{lem:Additive}
	The CPVL $\bfeta$ is additive, i.e.\ $\bfeta_t^{\eta_1\vee\eta_2}= \bfeta_t^{\eta_1} \vee\bfeta_t^{\eta_2}$ for all $t\geq 0$ with $\eta_1,\eta_2\in \N^V$.
\end{lemma}
\begin{proof}
	We can use the same line of arguments as in the proof of Lemma~\ref{lem:MonotonictyInRates}. This means that if $S$ is a jump time of either of the three processes and we assume that $\bfeta^{\eta_1\vee\eta_2}_{S-}= \bfeta^{\eta_1}_{S-}\vee\bfeta^{\eta_2}_{S-}$ holds, then it suffices to show that $\bfeta^{\eta_1\vee\eta_2}_{S}= \bfeta^{\eta_1}_{S}\vee\bfeta^{\eta_2}_{S}$. This follows by the definition of potential jumps; see \eqref{eq:JumpTimesOfCPVL}.
\end{proof}
\begin{lemma}\label{lem:MonotonMarkov}
	Let $\eta_1,\eta_2\in \N^V$ with $\eta_1\leq  \eta_2$. The CPVL $\bfeta$ is a monotone Markov process, i.e.\
	$\bfeta_t^{\eta_1}\leq  \bfeta_t^{\eta_2}$ for all $t\geq 0$.
\end{lemma}
\begin{proof}
	This is a immediate consequence of Lemma~\ref{lem:Additive}. 
\end{proof}

\subsection{Duality of CPVL and CPLI}
In this subsection we construct the CPLI in such a way that it is coupled with a CPVL such that the duality relation holds. The CPVL, denoted by $\bfeta$, has a constant infection rate, i.e.\ $\Lambda(n)=\lambda\in(0,\infty)$ for all $n\geq 0$. Recall that we assume $\birth,\death$ to satisfy Assumption~\ref{ass:RateAssumption}.

We use the same PPP's as in the construction of $\bfX^N$ to define the dual process. To be more precise fix $\xi\in\N^{V}$ and an arbitrary time $t\in \IR$. In the following we construct a process $\big(\widehat{\bfX}^N_{s,t}[\xi]\big)_{s\in (-\infty, t]}$ with initial value $\widehat{\bfX}^N_{t,t}[\xi](x)=\xi(x)$ for all $x\in V_N$. Note that in this case we let time run backwards. 

In the following we again suppress $N$ if it is clear by context. Similar as before we will now define jump times $(R_k)_{k\geq 0}$ and the changes in the process. First we set $R_0=t$ and $\widehat{\bfX}^N_{t,t}(\xi)=\xi$. Now assume that $R_{n-1}\geq s$ for some $n\geq 1$ and that $\big(\widehat{\bfX}^N_{u,t}[\xi]\big)_{u\leq R_{n-1}^{N}}$ is already defined, then we set
\begin{align*}
	\widehat{I}^N_{n}&:=\sup\{s\leq R_{n-1}: \exists ((x,y),(\alpha,s)) \in \II^N \text{ s.t. } \alpha<\lambda
	\},\\
	\widehat{U}^N_{n}(\xi)&:=\sup\{s\leq R_{n-1}: \exists (x,(\alpha,s)) \in \ID \text{ s.t. } \alpha< \death(\widehat{\bfX}^N_{ R_{n-1},t}[\xi](x)+1) \},\\
	\widehat{D}^N_{n}(\xi)&:=\sup\{s\leq R_{n-1}: \exists (x,(\alpha,s)) \in \IU \text{ s.t. } \alpha< \birth(\widehat{\bfX}^N_{R_{n-1},t}[\xi](x))\}.
\end{align*}
One notable difference is that $\widehat{I}^N_{n}$ does not depend on the initial configuration $\xi$, whereas the previous construction could also be modified in such a way for constant infection rates.

Then, set $R^N_{n}:= \widehat{I}^N_{n}\wedge \widehat{U}^N_{n}(\xi) \wedge \widehat{D}^N_{n}(\xi)$ and denote by $(\alpha_{n},x_{n})\in [0,\infty)\times V $ or $(y_{n},x_{n})\in \Vec{E}$ the associated point. Now we define the process after $R_{n-1}$ as follows. First we set $\bfxi_t\equiv \bfxi_{R_{n-1}}$ for all $t\in (R_{n-1},R_{n})$. Then, at time $R_{n}$ we set
\begin{align*}
	\widehat{\bfX}^N_{R_{n},t}[\xi](x_{n})=\begin{cases}
		0 &\text{ if } \quad ((x_{n},y_{n}),R_{n})\in \II \text{ and } \bfxi_{R_{n-1}}(y_{n})=0\\
		\widehat{\bfX}^N_{R_{n-1},t}[\xi](x_{n})+ 1 &\text{ if } \quad ((\alpha_{n},x_{n}),R_{n})\in \ID\\
		\widehat{\bfX}^N_{R_{n-1},t}[\xi](x_{n})- 1 &\text{ if } \quad (x_{n},R_{n})\in \IU\\
		\widehat{\bfX}^N_{R_{n-1},t}[\xi](x_{n}) &\text{ otherwise}
	\end{cases}
\end{align*}
and $\widehat{\bfX}^N_{R_{n},t}[\xi](y)=\widehat{\bfX}^N_{R_{n-1},t}[\xi](y)$ for all $y\neq x_{n}$.

Now by construction we again have $\widehat{\bfX}^N_{s,u}[\widehat{\bfX}^N_{u,t}[\xi]]=\widehat{\bfX}^N_{s,t}[\xi]$ where $s\leq u\leq t$. Unlike the previous construction, in this case we have $\widehat{\bfX}^N[\xi]\geq \widehat{\bfX}^M[\xi]$ for all $N\leq M$. Non-explosiveness of the finite system can be shown in the same way as before by coupling with a system of BD-Processes on $\N^V$ with rates $\death(\,\cdot\,+1)$ and $\birth(\cdot)$. This can be proven analogously as in Proposition~\ref{prop:CouplingWithIndBD}.

Finally, this Poisson construction can again be extend to infinite graphs. One can derive similar estimates as in Proposition~\ref{prop:BoundOnPotentialInfluence} for $\widehat{\bfX}_{s,t}$, which is, in fact, even easier, since the infection rate is no longer state dependent but constant. Then, one can show the analogous result of Theorem~\ref{thm:ExistenceAndFeller} and Proposition~\ref{prop:FiniteCoordinatesDomain}. This means that for every $s\in\IR$ and every $\xi\in \N^V$ the limit
\begin{equation}\label{eq:DualDoubleIndexedProcess}
	\widehat{\bfX}_{s,t}[\xi]=\lim_{N\to \infty}\widehat{\bfX}^N_{s,t}[\xi]
\end{equation}
exists. Furthermore, fix a $t\in \IR$, let $\bfxi_0$ be a $\N_0^V$ valued random variable independent of the PPP's $\II,\IU$ and $\ID$ used in the construction and set 
\begin{equation*}
	\bfxi_s:=\widehat{\bfX}_{(t-s)-,t}[\bfxi_0] \quad \text{ for }\quad s\geq 0.
\end{equation*}
Then, the process $\bfxi=(\bfxi_s)_{s\geq 0}$ is a Feller process. Note that by taking the left limit $\widehat{\bfX}_{(t-s)-,t}[\bfxi_0]$ we ensure that $\bfxi$ has right continuous paths.

For given $\xi\in \N_0^V$, denote again by $\widehat{\cI}(x)$ the potential infection times and by $\widehat{\cU}_{s,t}(x,\xi)$ and $\widehat{\cD}_{s,t}(x,\xi)$ the jumps up and down. In the following result we show that already by construction the CPLI is dual to the CPVL. This is also visualised in Figure~\ref{Xfig2-2}.

\begin{figure}[t]
	\centering 
	\subfigure{\includegraphics[width=67mm]{OnSitePPP.png}}\hfill
	\subfigure{\includegraphics[width=67mm]{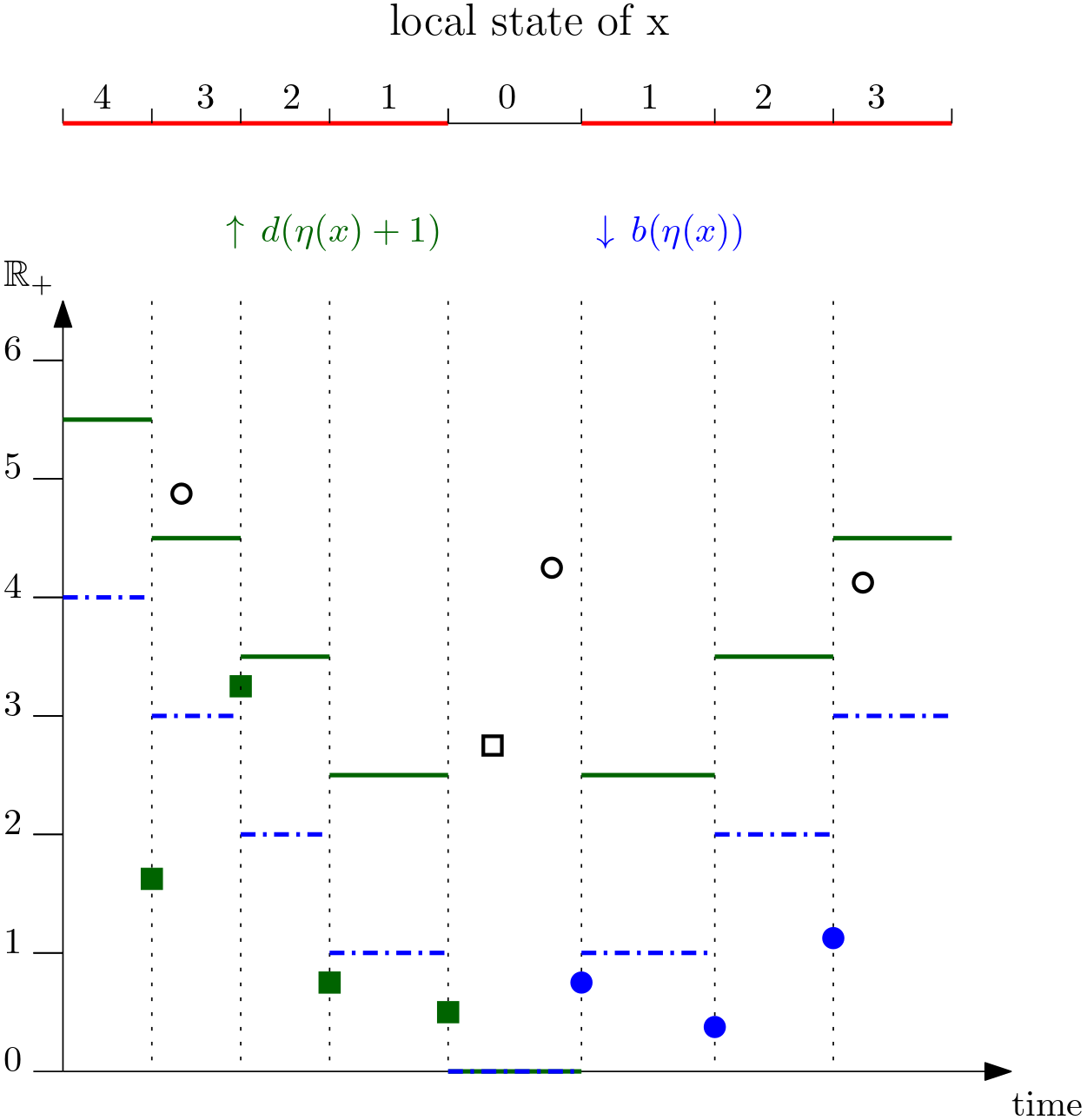}}
	\caption{Visualisation of the on site dynamic for both process with $\birth(n)=n$ and $\death(n)=n+\tfrac{1}{2}$ for all $n\geq 1$, where the symbols are the same as in Figure~\ref{fig:PoissonConstruction}. Note that the meaning of the squares and circles are reversed in the right picture compared to the left.
	}
	\label{Xfig2-2}
\end{figure}

\begin{theorem}\label{thm:DualityRelationFlow}
	Let $s<t$ and $\eta,\xi\in \N^{V} $. Let $(\bfX_{s,t}[\eta])_{t\geq 0}$ be the process defined in Theorem~\ref{thm:ExistenceAndFeller} and $(\widehat{\bfX}_{s,t}[\xi])_{t\geq 0}$ as in \eqref{eq:DualDoubleIndexedProcess}. If there exists a $u\in[s,t]$ such that $\bfX_{s,u}[\eta]\leq \widehat{\bfX}_{u+,t}[\xi]$, $\IP$-almost surely, then $\bfX_{s,u}[\eta]\leq \widehat{\bfX}_{u+,t}[\xi]$ for all $u\in[s,t]$, $\IP$-almost surely. This implies in particular that for any arbitrary $t\geq 0$ it holds that
	\begin{equation*}
		[0,t]\ni s\mapsto\IP\big(\bfeta^{\eta}_{s}\leq \bfxi^{\xi}_{t-s}\big)
	\end{equation*}
	is a constant function.
\end{theorem}
\begin{proof}
	Let us again denote by $S$ a jump time of either of the two processes and $x\in V$ the unique associated vertex which is affected by the change. It suffices to show that 
	\begin{equation*}
		\bfX_{s,S-}[\eta]\leq \widehat{\bfX}_{S,t}[\xi] \text{ if and only if } \bfX_{s,S}[\eta]\leq \widehat{\bfX}_{S+,t}[\xi].
	\end{equation*}
	
	We first assume that
	\begin{equation}\label{eq:Contradiction1DualityProof}
		\bfX_{s,S-}[\eta]\not\leq \widehat{\bfX}_{S,t}[\xi] \text{ and } \bfX_{s,S}[\eta]\leq \widehat{\bfX}_{S+,t}[\xi]
	\end{equation}
	and show that this situation cannot occur. We will show this with a proof by contradiction. 
	\begin{itemize}
		\item  Suppose that $S\in \IU_x$. Then, the only relevant case is if $\bfX_{s,S-}[\eta](x)-1=\widehat{\bfX}_{S,t}[\xi](x)$, $S\in \widehat{\cD}_{s,t}(x,\xi)$ and $S\notin \cU_{s,t}(x,\eta)$. We have  $\bfX_{S-,t}[\eta](x)=\bfX_{S,t}[\eta](x)$, and thus there exists $\alpha\in [0,\infty)$ such that $(\alpha,S) \in \IU_x$  and $\alpha< \birth(\bfX_{s,S-}[\eta](x)+1)$, which implies $S\in \widehat{\cU}_{s,t}(x,\xi)$. This contradicts \eqref{eq:Contradiction1DualityProof}.
		\item Next suppose that $S\in \ID_x$. Similarly as in the previous point, the only relevant case is $\bfX_{s,S-}[\eta](x)-1=\widehat{\bfX}_{S,t}[\xi](x)$,  $S\in \cD_{s,t}(x,\eta)$ and $S\notin \widehat{\cU}_{s,t}(x,\xi)$. But, this implies that $\widehat{\bfX}_{S,t}[\xi](x)=\widehat{\bfX}_{S+,t}[\xi](x)$, and thus there exists $\alpha\in [0,\infty)$ such that $(\alpha,S) \in \IU_x$  and $\alpha< \death(\widehat{\bfX}_{s,S+}[\eta](x)+1)$, which implies $S\in \widehat{\cU}_{s,t}(x,\xi)$. Thus, we  have  again a contradiction to \eqref{eq:Contradiction1DualityProof}.
		\item Finally, assume that $S\in \II_{(x,y)}$ for some $y\in \cN_x$. First, we observe that $\bfX_{S-,t}[\eta](x)=\bfX_{S,t}[\eta](x)$ and $\widehat{\bfX}_{S,t}[\xi](x)=\widehat{\bfX}_{S+,t}[\xi](x)$. Furthermore, the only relevant case is if $\widehat{\bfX}_{S+,t}[\xi]\geq \bfX_{S,t}[\eta](x)>0$. It follows immediately that $\bfX_{S,t}[\eta](y)>0$, and thus we know by assumption that $\widehat{\bfX}_{S+,t}[\xi](y)>0$. This allows us to conclude that $\widehat{\bfX}_{S+,t}[\xi](x)=\widehat{\bfX}_{S,t}[\xi](x)\geq \bfX_{S,t}[\eta](x)$, and thus $\bfX_{s,S-}[\eta]\leq \widehat{\bfX}_{S,t}[\xi]$. This leads to a contradiction of \eqref{eq:Contradiction1DualityProof}.
	\end{itemize}
	On the other hand, one can show that 
	\begin{equation*}
		\bfX_{s,S-}[\eta]\leq \widehat{\bfX}_{S,t}[\xi] \text{ and } \bfX_{s,S}[\eta]\not\leq \widehat{\bfX}_{S+,t}[\xi]
	\end{equation*}
	cannot be true by the same line of arguments. 
\end{proof}
\begin{proof}[Proof of Proposition~\ref{prop:Duality}]
	This is a direct consequence of Theorem~\ref{thm:DualityRelationFlow}.
\end{proof}
Finally, we state the monotonicity properties of the CPLI.
\begin{lemma}\label{lem:MonotonictyForTheCPLI}
	Let $\bfxi=(\bfxi_t)_{t\geq 0}$ be a CPLI with infection rates $\lambda>0$ and rate function $\birth,\death$.
	\begin{enumerate}
		\item  $\bfxi$ is a monotone Markov process, i.e.\ if $\xi_1\leq \xi_2$, then $\bfxi_t^{\xi_1}\leq \bfxi_t^{\xi_2}$ for all $t\geq 0$, where $\xi_1,\xi_2\in \N^V$.
		\item $\bfxi$ is a monotone with respect to $\lambda$, i.e.\ if $\lambda<\lambda'$ then there exists a CPVL $\bfxi'$ with infection rate $\lambda'$ and rate functions $\birth$ and $\death$ and $\bfxi_{0}\leq \bfxi'_{0}$ such that $\bfxi_{t}\leq  \bfxi'_{t}$ for all $t\geq 0$.
	\end{enumerate}
\end{lemma}
\begin{proof}
	This can be proven analogously as Lemma~\ref{lem:MonotonictyInRates} and Theorem~\ref{thm:DualityRelationFlow}, i.e.\ by showing that all potential jumps preserver the order.
\end{proof}

\section{Survival of the CPVL and its upper invariant law}\label{sec:SurvivalCPVL}
The first result we will show in this section is that the critical behaviour of the survival probability of the CPVL does not depend on the initial configuration as long as the initial configuration $\eta$ is non-empty and finite. This means that 
\begin{equation*}
	\IP(|\bfeta^{\delta_{\zero}}_t|>0 \,\forall\, t\geq 0)>0\,\, \Leftrightarrow\,\,\IP(|\bfeta^{\eta}_t|>0 \,\forall\, t\geq 0)>0
\end{equation*}
for all $\eta\in \N_0^V$ with $|\eta|\in (0,\infty)$.
\begin{proof}[Proof of Propostion~\ref{prop:CriticalValueIndependentOfStart}]
	If the left-hand side is true, then by monotonicity, shown in Lemma~\ref{lem:MonotonMarkov}, and translation invariance the right-hand side is a direct consequence. On the other hand, by additivity, shown in Lemma~\ref{lem:Additive}, we know that
	\begin{equation*}
		\bfeta^{\eta}_t=\bigvee_{x\in V: \eta(x)>0} \bfeta_t^{\eta(x)\delta_x}
	\end{equation*}
	for all $t\geq 0$. In particular it holds that if $\bfeta^{\eta}$ has a strictly positive probability to survive so does $\bfeta^{M \delta_x}$, where $M:=\max_{x\in V}\eta(x)<\infty$. 
	Now we see that
	\begin{equation*}
		\IP(\bfeta^{\delta_x}_1=M\delta_x)>0
	\end{equation*}
	and together with the the Markov property we conclude that
	\begin{equation*}
		\IP(|\bfeta^{\delta_x}_t|>0 \, \forall\,  t\geq 0)\geq \IP(|\bfeta^{M\delta_x}_t|>0\,\forall\, t\geq 0)\IP(\bfeta^{\delta_x}_1=M\delta_x)>0.\qedhere
	\end{equation*}
\end{proof}
Now we prove the sufficient criterion stated in Proposition~\ref{thm:upperbound} for survival of the CPVL in the sense that the survival probability is positive.
\begin{proof}[Proof of Proposition~\ref{thm:upperbound}]
	Recall that $\bfeta$ is a CPVL constructed via the Poisson construction given in Section~\ref{sec:Construction} with an infection rate function $\Lambda$ and rate functions $\birth$ and $\death$. According to Proposition~\ref{lem:MonotonictyInRates} we can couple $\bfeta$ with another CPVL $\bfeta'$ with the same $\Lambda$ and $\death$, but a different birth rate function $b'\equiv 0$, i.e.\ $\birth'(n)=0$ for all $n\geq 0$, such that $\bfeta_t \geq \bfeta'_t$ for all $t\geq0$, since $\birth(n)\geq \birth'(n)$ for all $n\geq 0$. This implies in particular that if $\bfeta'$ survives so does $\bfeta$.
	
	The choice of $\birth'$ implies that for the process $\bfeta'$ we do not allow any birth events on a vertex. Therefore, at infection, a vertex gets assigned the viral load $1$, which cannot increase further and will eventually drop to $0$ at rate $\death(1)$. This means that $\bfeta'$ has the same dynamics as a classical contact process with infection rate $\Lambda(1)$ and recovery rate $\death(1)$. Thus, by definition of the critical value $\lambda_c^{CP}$ it follows that  $\bfeta'$ survives with positive probability if $\Lambda(1)/d(1)>\lambda_c^{CP}$.
\end{proof}
Next, we prove the criterion that guarantees extinction stated in Theorem~\ref{thm:LowerBound}. 
\begin{proof}[Proof of Theorem~\ref{thm:LowerBound}]
	We consider the situation where $\bfeta_{0}=\delta_{\zero}$, i.e.\ the origin is infected and in state $1$. The aim of the proof is to couple the CPVL with a discrete time branching process such that if the branching process dies out so does the CPVL. Let us introduce the following notation
	\begin{equation}
		\T:=\{\alpha=0 \alpha_1\cdots \alpha_n : n\in \N,\ a_i\in \N, \ \forall \  i\leq n\}.
	\end{equation}
	The idea of this coupling is to count every infection event which could potentially transmit an infection. This means we count every $((x,y),t)\in \II$ with $\bfeta_{t-}(x)>0$ regardless of the state of the neighbour $y$.
	
	Thus, we change our perspective in the sense that we say that we start initially with a particle at the origin $\zero$, which is born at time $0$ with initial viral load $1$ and the viral load evolves according to a BD-process $X(0)=(X_t(0))_{t\geq 0}$. This particle has a life time $T_0:=\inf\{t>0: X_t(0)=0\}\stackrel{d}{=} \tau_{\text{rec}}$ and we lable it by $\bfX_{0}:=(\zero,0,X(0))$. Let us denote by  
	\begin{equation*}
		\cI_{0}:=\{(t,y) \in [0,T_0)\times \cN_{\zero}: \exists \alpha\in [0,\infty)\text{ s.t. }  \big((0,y),(\alpha,t)\big) \in \II \text{ and } \alpha<\Lambda(X_t(0)) \}
	\end{equation*}
	the set of all birth times and location of all particles generated by the initial particle $\bfX_0$.
	It is not difficult to see that the total number of generated particles $|\cI_{0}|\stackrel{d}{=} \text{Poi}\big(D\int_{0}^{\tau_{\text{rec}}}\Lambda(X_s)ds\big)$, where $X$ is a BD-process with rates $\birth$ and $\death$ and $X_0=1$. We can now order the points in $\cI_{0}$ according to the birth time such that we can represent the set as
	\begin{equation*}
		\cI_{0}:=\{(x_{01},t_{01}),\dots (x_{0|\cI_{0}|},t_{0|\cI_{0}|})\}
	\end{equation*}
	where $0\leq t_{01}\leq \cdots\leq t_{0|\cI_{0}|}<T_0$ and $x_{0k}\in N_{\zero}$ for all $k\in \{1,\cdots, |\cI_{0}|\}$. Note that since we assumed that $\IE[\int_{0}^{\tau_{\text{rec}}}\Lambda(X_s)ds]<\infty$ there are almost surely only finitely many space-time points.
	
	The only piece which is missing, is the evolution of the offspring particles. These are again started with an initial viral load of $1$ which evolve according to independent BD-process $X(0k)$ with rate functions $\birth$ and $\death$ and we denote the life time by $T_{0k}$. 
	
	For the sake of the coupling with $\bfeta$ we check for every pair $(x_{0k},t_{0k})$ if $\bfeta_{t_{0k}-}(x_{0k})=0$. If that is the case then we can choose $X(0k)=\bfeta_{t_{0k}}(x_{0k})$, since $t_{0k}$ is indeed a true infection time for CPVL such that the evolution of this particle is coupled via the Poisson construction with the viral load of vertex $x_k$ started at time $t_k$. If that is not the case, i.e.\ $\eta_{t_k}(x_k)>0$, then we choose $X(0k)$ independent from the Poisson construction and the other particles. We can interpret the second case as if we generate an artificial particle. 
	
	We again denote all offspring particles by $\bfX_{0k}:=(x_{0k},t_{0k},X(0k))$ and $\cI_{0k}$ again denotes the set of all birth times and location of all particles generated by the $0k$th-particle $\bfX_{0k}$.
	
	According to this scheme we successively generate a family of random variables $\bfX_{\alpha}=(x_{\alpha},t_{\alpha},X(\alpha))$ and random sets $\cI_{\alpha}$. Of course if the $\alpha$-th particle was never created then $\cI_{\alpha}=\emptyset$. With this in place we define by
	\begin{equation*}
		\bfzeta_s(x):=\1_{\{x_{0}=x, s\in [0,T_{0})\}}+\sum_{\alpha\in \T}\1_{\{\cI_{\alpha}\neq \emptyset\}}\sum_{k=1}^{|\cI_{\alpha}|}\1_{\{x_{\alpha k}=x, s\in [t_{\alpha k},T_{\alpha k}+t_{\alpha k})\}}
	\end{equation*}
	the number of all particles alive at time $s$ at location $x$. Note that by construction $\bfeta_s\leq \bfzeta_s$ for all $s\geq 0$. Thus, if we can show that $|\bfzeta_t|= 0$ at some time $t$, then this implies extinction of $\bfeta$.
	
	Since we assumed that $\IE[\int_{0}^{\tau_{\text{rec}}}\Lambda(X_s)ds]<\infty$ we know that every particle will die almost surely. Thus, $|\bfzeta_t|= 0$ at some time $t$ is equivalent that the total progeny $\cT:=1+\sum_{\alpha\in\T}|\cI_{\alpha}|<\infty$ almost surely. Now if we forget about the locations and precise times when the particles are alive, then $\cT$ is nothing else but the total progeny of a discrete time branching process with offspring distribution $Y\stackrel{d}{=} \text{Poi}( D \int_{0}^{\tau_{\text{rec}}}\Lambda(X_s)ds)$ and we know that if $\IE[Y]= D \IE[\int_{0}^{\tau_{\text{rec}}}\Lambda(X_s)ds]<1$, then $\cT<\infty$ almost surely, which proves the claim.
\end{proof}
Finally, we show the existence of an upper invariant law $\overline{\nu}$ if $\IE[\tau_{\text{rec}}]<\infty$.
\begin{proof}[Proof of Theorem~\ref{thm:InvariantDistributionCPVL}]
	Let us denote by $(T_t)_{t\geq 0}$ the semigroup associated with $(\bfeta_t)_{t\geq 0}$ and by $(S_t)_{t\geq 0}$ the semigroup associated with the Markov process $(\bfzeta_t)_{t\geq 0}$ with transitions 
	\begin{equation*}
		\begin{aligned} 
			\zeta(x)&\to \zeta(x)+1\quad \text{ at rate } 
			\birth(\eta(x))
			\text{ and }\\
			\zeta(x)&\to \zeta(x)-1\quad \text{ at rate } \death(\eta(x))\1_{\{\eta(x)\geq 2\}}.
		\end{aligned}
	\end{equation*}
	Appealing to Proposition~\ref{prop:CouplingWithIndBD} we can couple these two process such that if $\bfeta_0\leq \bfzeta_0$, then $\bfeta_t\leq \bfzeta_t$ for all $t\geq 0$. Note that the state space of $\bfzeta$ is $\N^V$.
	
	Furthermore, let $f:\overline{\N}_0^V\to \IR$ be a continuous and increasing function, i.e.\ $f(\eta^1)\leq f(\eta^2)$ if $\eta^1\leq \eta^2$. By the coupling between $(\bfeta_t)_{t\geq 0}$ and $(\bfzeta_t)_{t\geq 0}$ we know that $T_tf(\eta)\leq S_tf(\eta)$ for all $\eta\in \N^V$ and $t\geq 0$ if $f$ is increasing. Note that since both systems are monotone it follows that $T_tf$ and $S_tf$ are again increasing function. First of all, it follows that
	\begin{equation*}
		\int f(x) \pi T_{t+s}(d x)=\int T_{s}[T_{t}f(x)] \pi(d x)\leq \int S_{s}[T_{t}f(x)] \pi(d x)=\int T_{t}f(x) \pi S_{s}(d x).
	\end{equation*}
	Recall that $\pi$ was the invariant distribution of $(\bfzeta_t)_{t\geq 0}$, i.e.\ $\pi S_s=\pi$ for all $s\geq 0$. Thus, it follows that
	\begin{equation*}
		\int f(x) \pi T_{t+s}(d x)\leq \int f(x) \pi  T_{t}(d x).
	\end{equation*}
	This means that $t\mapsto \int f(x) \pi  T_{t}(d x)$ is a decreasing function for all increasing functions $f$. Furthermore, this function is bounded from below by $0$ such that it follows that
	\begin{equation*}
		\overline{\nu}f:=\lim_{t\to \infty}\int f(x) \pi  T_{t}(d x)
	\end{equation*}
	exists for all increasing $f$. Note that the class of all bounded and increasing functions is convergence determining.
	
	Next let $\mu$ be another probability law on $\N^V$ such that $\pi\prec \mu$. By monotonicity it follows that $\pi T_t f\leq \mu T_t f$ for all $t\geq 0$, which implies
	\begin{equation*}
		\overline{\nu}f\leq \liminf_{t\to \infty} \int T_s f(x) \mu T_{t}(d x).
	\end{equation*}
	Monotonicity also implies that
	\begin{equation*}
		\limsup_{t\to \infty} \int T_s f(x) \mu T_{t}(d x)\leq \lim_{t\to \infty} \int T_s f(x) \mu  S_{t}(d x)= \pi T_s f,
	\end{equation*}
	where we also used that $\pi$ is the unique invariant (limiting) distribution of $(S_t)_{t\geq 0}$. But now we can use the semigroup property of $(T_t)_{t\geq 0}$, i.e.\ that $T_tT_s=T_{t+s}$ and that $t\mapsto T_t$ is strongly continuous and send $s\to \infty$. This yield that 
	\begin{equation*}
		\overline{\nu}f\leq \lim_{t\to \infty} \int f(x) \mu T_{t}(d x)\leq \overline{\nu}f.
	\end{equation*}
	Thus, $\overline{\nu}$ is the largest invariant distribution with respect to the stochastic order.
	
	It remains to show the two additional statements. First, we know that $\pi(\N^V)=1$ and by construction $\overline{\nu}\preceq \pi$. Thus it follows that $\overline{\nu}(\N^V_0)=1$.
	
	For the second statement we note that $p:=\overline{\nu}(\{\szero\})\in \{0,1\}$ follows by the fact that $\nu(\,\cdot\, ):=\overline{\nu}(\,\cdot\, |\{\szero\})$ is again an invariant distribution and 
	\begin{equation*}
		\overline{\nu}=(1-p)\nu+p \delta_{\szero}.
	\end{equation*}
	If we assume that $p<1$, then it follows that $\overline{\nu}\preceq \nu$, but since $\overline{\nu}$ is the upper invariant law it must hold that $\overline{\nu}= \nu$, which yields that $p=0$.
\end{proof}

\begin{proof}[Proof of Corollary~\ref{Cor:SurvivalForSpecialCase}]
	Recall that we assumed $\Lambda(n)=\lambda n^{\gamma}$ for all $n\geq 0$, where $\lambda,\gamma>0$. Then, we denoted by $\lambda_c(\gamma)$ the corresponding critical value for survival. A direct consequence of Proposition~\ref{thm:upperbound} is that $\lambda_c(\gamma)<\infty$ for all $\gamma\geq0$ in both cases.
	
	For the other direction we first prove that
	\begin{equation}\label{eq:EquivalenceForSurvival}
		\IE\bigg[\int_{0}^{\tau_{\text{rec}}}\Lambda(X_s)\mathrm{d}s\bigg]<\infty \Leftrightarrow \sum_{n=1}^{\infty} \frac{n^\gamma}{\death(n)} \prod_{j=1}^{n-1}\frac{\birth(j)}{\death(j)}<\infty.
	\end{equation}
	For $\gamma=0$ this is equivalent to $\IE[\tau_{\text{rec}}]<\infty$, which we will assume from now on.
	
	Let $Z$ be BD-process on $\N$ with $Z_0=1$ and the same rates as $X$, except that it cannot enter the state $0$, i.e.\ it has transitions
	\begin{align*}
		n\to n+1 &\text{ with rate } \birth(n)\\
		n\to n-1 &\text{ with rate } \death(n)\1_{\{n\geq 2\}}.
	\end{align*}
	Since we assumed $\IE[\tau_{\text{rec}}]<\infty$, we know that the process $Z$ is positive recurrent. Thus, we have 
	\begin{equation*}
		\pi(k)=\pi(1)\frac{\birth(1)\cdots \birth(k-1)}{\death(2)\cdots \death(k)} \quad \text{ with }\quad \pi(1):=\bigg(1+\sum_{n=1}^{\infty}\prod_{i=1}^{n} \frac{\birth(i)}{\death(i+1)}\bigg)^{-1}.
	\end{equation*}
	We define $T_1:=\inf\{t\geq 1: Z_t=1\}$. Then, one can show that 
	\begin{equation*}
		\IE\bigg[\int_{0}^{\tau_{\text{rec}}}\Lambda(X_s)\mathrm{d}s\bigg]=\frac{\birth(1)+\death(1)}{\death(1)}\IE\bigg[\int_{0}^{T_1}\Lambda(Z_s)\mathrm{d}s\bigg] 
	\end{equation*}
	using that the number of times needed to enter state $0$ from $1$ is geometric distributed  with success probability $\frac{\death(1)}{\birth(1)+\death(1)}$. 
	
	Next, we use that $\pi(y)=\IE\Big[\int_{0}^{T_1} \1_{\{Z_t=y\}}dt\Big]$ for all $y\in V$.The right-hand side is a stationary distribution of $Z$. This can be demonstrated by performing a time shift in the integral and applying the strong Markov property. Moreover, we also note that $T_1<\infty$ almost surely.	See the proof of \cite[Proposition~2.59]{liggett2010continous} for more details. Moreover, the equality follows from the fact that the stationary distribution is unique. Thus, we get that
	\begin{equation*}
		\IE\bigg[\int_{0}^{T_1}\Lambda(Z_t)dt\bigg]=\sum_{k=1}^{\infty}\Lambda(k)\IE\bigg[\int_{0}^{T_1} \1_{\{Z_t=k\}}dt\bigg]=\pi(1)\sum_{k=1}^{\infty}\Lambda(k)\prod_{i=1}^{k-1} \frac{\birth(i)}{\death(i+1)}.
	\end{equation*}
	Putting everything together yields that
	\begin{equation*}
		\IE\bigg[\int_{0}^{\tau_{\text{rec}}}\Lambda(X_s)\mathrm{d}s\bigg] \Leftrightarrow \sum_{n=1}^{\infty} \frac{n^\gamma}{\death(n)} \prod_{j=1}^{n-1}\frac{\birth(j)}{\death(j)},
	\end{equation*}
	which proves the claim.
	
	Now we just need to verify \eqref{eq:EquivalenceForSurvival} for both cases. If this is done, then by Theorem~\ref{thm:upperbound} we get that for $\lambda$ large enough the CPVL dies out, and thus $\lambda_c(\gamma)>0$.
	
	First set $\birth(n)=n$ and $\death(n)=(n+(a-1))\1_{\{n\geq 1\}}$ for all $n\geq 0$ with $a\geq 1$, then we get that
	\begin{equation*}
		\sum_{n=1}^{\infty} \frac{n^\gamma}{n+a-1} \prod_{j=1}^{n-1}\frac{j}{j+a-1}=  \sum_{n=1}^{\infty} \Gamma(a+1)n^{\gamma-1}\frac{\Gamma(n+1)}{\Gamma(n+a+1)}<\infty
	\end{equation*}
	if $\gamma<a-1$, since $\lim_{n\to \infty}\frac{\Gamma(n+a)}{\Gamma(n)}\frac{1}{n^a}=1$.

	Next let $\birth(n)=\alpha n$ and $\death(n)=\beta n$ for all $n\geq 0$ with $\alpha<\beta$, then we get for all for all $\gamma\geq 0$ that 
	\begin{equation*}
		\sum_{n=1}^{\infty} \frac{n^\gamma}{\beta n} \prod_{j=1}^{n-1}\frac{\alpha}{\beta}=  \sum_{n=1}^{\infty} n^{\gamma-1}\alpha\bigg(\frac{\alpha}{\beta}\bigg)^{n}<\infty.\qedhere
	\end{equation*}
\end{proof}

\section{Phase transition of the CPLI}\label{Sec:PhaseTransitionCPLI}
Throughout the section, we consider the CPVL with a constant infection rate function, i.e.\ $\Lambda(n)=\lambda$ for all $n\geq 0$, and assume that $\supp(b)=\supp(d)=\N$. Recall that we denote the critical infection for survival by $\lambda_c$. 

The main goal of this section is to prove Theorem~\ref{thm:PhasetrastionCPLI}, and therefore to establish the connection between the phase transition of the CPLI and the CPVL. But in order to do so we need some technical results about survival probability of the CPVL.
\begin{lemma}\label{lem:AsymtoticSurvivalAS}
	Assume that $\lambda>\lambda_c$ and let $(M_N)_{N\geq 1}$ be a monotone sequences of subsets that convergence to $V$, i.e. $M_N\subset V$ such that $M_N\uparrow V$. Then it holds that 
	\begin{equation*}
		\lim_{N\to \infty}\IP(|\bfeta^{\delta_{M_N}}_t|>0 \,\forall \, t\geq 0)=1.
	\end{equation*}
\end{lemma}
\begin{proof}
	First, note that we assumed that $G=\Gamma(\mathfrak{g},\mathfrak{s})$ is a Cayley graph with $\mathfrak{g}$ being an infinite group that is finitely generated by $\mathfrak{s}$. Thus, there must exist a graph automorphism $S$ and a vertex $x_0\in V$ such that $|\{S^n(x_0): n\in \Z\}|=\infty$.
	Without loss of generality, we set $x_0:=\zero$ otherwise we relabel. Furthermore, we set $\cS:=\{S^n(\zero): n\in \Z\}$.
	
	Now define $Y_x:=\1_{\{|\bfeta^{x}_t|>0\,\forall t\geq 0\}}$ as the indicator variable for the event of survival of the process, which starts with only $x$ being in state $1$ and everything else $0$. 
	
	Recall that we constructed $\bfeta$ through three independent PPP $\II,\IU$ and $\ID$. From a different perspective these PPP's can be seen as families of independent Poisson point processes on $\IR\times \IR_+$, namely $(\II_{(x,y)})_{(x,y)\in \vec{E}}$, $(\IU)_{x\in V}$ and $(\ID)_{x\in V}$ on $\IR\times \IR_+$. To shorten the notation, we define $\Delta:=(\II,\IU,\ID)$.
	
	It is clear that $S$ is a measure preserving map with respect to the distribution of this family of the three Poisson processes, i.e.\ 
	\begin{equation*}
		S^{-1}(\II,\IU,\ID)=:\big(S^{-1}(\II),S^{-1}(\IU),S^{-1}(\ID)\big)\stackrel{d}{=}(\II,\IU,\ID).
	\end{equation*}
	Furthermore, every of the three processes are a collection of independent random variables, and thus are ergodic with respect to $S$. Since the three processes are independent it follows immediately that $(\Delta,S)$ is ergodic.
	
	Since $\bfeta$ is constructed via the Poisson construction we see that there must exist a measurable function $f$ mapping $\Delta$ to $\{0,1\}$ such that
	\begin{align*}
		f(S^{-k}(\Delta))=\1_{\{|\bfeta^{x_k}_t|>0\,\forall t\geq 0\}}:=Y_k
	\end{align*}
	for every $k\in\Z$, where $x_k=S^{-k}(\zero)$. Note that by translation invariance, $\IP( Y_0 = 1 )=\IP(Y_{k} = 1)$ for all $k\in \Z$ and by assumption we know that $\IP(Y_{\zero} = 1)>0$. Then by Birkhoff's mean ergodic theorem \cite[Theorem~9.6]{kallenberg2002foundations}, it follows that
	\begin{align*}
		\frac{1}{2n}\sum_{k=-n}^{n}Y_k =\frac{1}{2n}\sum_{k=-n}^{n}f(S^{-k}(\Delta))\to\IE[f(\Delta)]=\IP\big(|\bfeta^{\delta_{\zero}}_t|>0\,\forall\, t\geq 0\big)>0
	\end{align*}
	almost surely. This implies that almost surely there must exist a $k$ for which $Y_k= 1$. By additivity, it follows that the event $\big\{|\bfeta^{\delta_{\cS}}_t|>0 \,\,\forall t\geq 0\big\}$ occurs as soon as the event $\{Y_k = 1\}$ occurs for some vertex in $x_k\in \cS\subset V$ which proves the statement.
\end{proof}

\begin{lemma}\label{lem:AysmtoticSurvivalAtOneSite}
	Assume that $\lambda>\lambda_c$. Then it holds that 
	\begin{equation*}
		\lim_{N\to \infty}\IP\big(|\bfeta^{N\delta_{\zero}}_t|>0 \, \forall \,t\geq 0\big)=1.
	\end{equation*}
\end{lemma}
\begin{proof}
	First set $T_0^N:=\inf\{t>0:\bfeta_t^{N\delta_{\zero}}(\zero)=0\}$. A comparison with a pure death process with rate function $\death$ yields that $\IE[T_0^N]\geq \sum_{n=1}^{N}\tfrac{1}{d(n)}$. Therefore, it holds that $\IE[T_0^N]\to \infty$ as $N\to \infty$, since $\death\in O(n)$. Now, \cite[Proposition~7.10]{cattiaux2009quasi} implies that $T_0^N\to \infty$ almost surely.
	
	
	Suppose the statement is false, which means that there exists an $\varepsilon>0$ such that
	\begin{equation*}
		\lim_{N\to \infty}\IP\big(|\bfeta^{N\delta_{\zero}}_t|>0 \, \forall \,t\geq 0\big)< 1-\varepsilon<1.
	\end{equation*}
	Then, by monotonicity follows that $\IP\big(|\bfeta^{N\delta_{\zero}}_t|>0 \, \forall \,t\geq 0\big)< 1-\varepsilon$
	for all $N>0$. On the other hand, it follows by Lemma~\ref{lem:AsymtoticSurvivalAS} that there exists a $\eta_\varepsilon$ such that
	\begin{equation*}
		\IP\big(|\bfeta^{\eta_{\varepsilon}}_t|>0 \,\forall\, t\geq 0\big)>1-\frac{\varepsilon}{3}.
	\end{equation*}
	It again follows by monotonicity that on the event that $\bfeta_t(\zero)>0$ it holds that
	\begin{equation*}
		\IP(\bfeta_{t+1}\geq \eta_{\varepsilon}|\bfeta_{t})\geq \IP(\bfeta_{t+1}\geq \eta_{\varepsilon}|\bfeta_{t}=\delta_{\zero})>\varepsilon',
	\end{equation*}
	where $\varepsilon'$ only depends on $\eta_{\varepsilon}$ and the left hand side is the probability that we reach a larger state than $\eta_{\varepsilon}$ after a time period of length $1$. 
	
	In the next step we can use the strong Markov property to conclude that
	\begin{align*}
		\IP\big(|\bfeta^{N\delta_{\zero}}_t|<|\eta_{\varepsilon}| \,\forall\, t\in \{1,\dots, \lfloor T_0^{N}\rfloor\}\big)&<\IE\bigg[\prod_{i=1}^{\lfloor T_0^{N}\rfloor}\IP\big(|\bfeta^{N\delta_{\zero}}_{i}|<|\eta_{\varepsilon}|\,\big|\bfeta^{N\delta_{\zero}}_{i-1}\big)\bigg]\\
		&<\IE\big[(1-\varepsilon')^{\lfloor T_0^{N}\rfloor}\big].
	\end{align*}
	Since we know that $T_0^N\to \infty$ almost surely as $N\to \infty$ we obtain with monotone convergences that $\IE\big[(1-\varepsilon')^{\lfloor T_0^{N}\rfloor}\big]\to 0$. Thus, there exists an $N_0\geq 0$ such that $\IE\big[(1-\varepsilon')^{\lfloor T_0^{N}\rfloor}\big]<\frac{\varepsilon}{3}$ for all $N\geq N_0$
	\begin{equation*}
		\begin{aligned}
			\IP(\exists t\leq T_0^N \text{ s.t. } \bfeta^{N\delta_{\zero}}_t\geq \eta_{\varepsilon})&\geq \IP(\bfeta^{N\delta_{\zero}}_t\geq \eta_{\varepsilon} \text{ for some } t\in \{1,\dots, \lfloor T_0^N\rfloor \big)\\
			&\geq 1-\IE\big[(1-\varepsilon')^{\lfloor T_0^{N}\rfloor} \big]>1-\frac{\varepsilon}{3}.
		\end{aligned}
	\end{equation*}
	However, by using again the strong Markov property it follows that
	\begin{equation*}
		\IP\big(|\bfeta^{N\delta_{\zero}}_t|>0 \forall t\geq 0\big)>\Big(1-\frac{\varepsilon}{3}\Big)^2>1-\varepsilon,
	\end{equation*}
	which leads to a contradiction.
\end{proof}

Finally, we can show the desired result and establish the formal connection between the two phase transitions.

\begin{proof}[Proof of Proposition~\ref{thm:PhasetrastionCPLI}]
	Let us assume that $\lambda\neq \lambda_c$. Thus, by Lemma~\ref{lem:AysmtoticSurvivalAtOneSite} we know that 
	\begin{equation*}
		\lim_{N\to \infty}\IP(|\bfeta^{N\delta_y}_t|>0 \,\forall\, t\geq 0)=\begin{cases}
			1   &\text{ if } \lambda>\lambda_c\\
			0   &\text{ if } \lambda<\lambda_c.
		\end{cases}
	\end{equation*}
	By using the duality relation stated in Proposition~\ref{prop:Duality} and monotonicity of the measure $\overline{\mu}$ we get that
	\begin{equation*}
		\lim_{N\to \infty}\IP(|\bfeta^{N\delta_y}|>0 \, \forall \, t\geq 0)=\lim_{N\to \infty}\overline{\mu}(\xi: \xi(y)< N)=\overline{\mu}(\xi: \xi(y)< \infty).
	\end{equation*}
	Now consider a sequences $(V_n)_{n\in \N}$ of finite subset of $V$, such that $V_n\subset V_{n+1}$ and $\bigcup_{n\in \N}V_n=V$. Then we see that for $\lambda>\lambda_c$ 
	\begin{equation*}
		\overline{\mu}(\xi: \xi(y)<\infty \,\forall\, y\in V_n )= \overline{\mu}\bigg(\bigcap_{y\in V_n}\{\xi:  \xi(y)<\infty\}\bigg)=1
	\end{equation*}
	for every $n\geq 0$, where we used the fact that the intersection of finitely many sets with mass $1$ has again mass $1$. By using continuity from above we get that
	\begin{equation*}
		\overline{\mu}(\xi: \xi(y)<\infty \,\forall\, y\in V)= \lim_{n\to 0} \overline{\mu}\,\bigg(\bigcap_{y\in V_n}\{\xi:  \xi(y)<\infty\}\bigg)=1.
	\end{equation*}
	Analogously it follows for $\lambda<\lambda_c$
	\begin{equation*}
		\overline{\mu}(\xi: \exists x\in V_n \text{ s.t. } \xi(y)<\infty)= \overline{\mu}\,\bigg(\bigcup_{y\in V_n}\{\xi:  \xi(y)<\infty\}\bigg)=0
	\end{equation*}
	and by using continuity from below we get that
	\begin{equation*}
		\overline{\mu}(\xi:  \exists x\in V \text{ s.t. } \xi(y)<\infty)= \lim_{n\to 0} \overline{\mu}\,\bigg(\bigcup_{y\in V_n}\{\xi:  \xi(y)<\infty\}\bigg)=0.\qedhere
	\end{equation*}
\end{proof}
\textbf{Acknowledgements:} I would like to thank Daniel Valesin and Michel Reitmeier for insightful discussions and for their valuable remarks and suggestions regarding the research presented here. I would also like to thank the anonymous referees for their detailed comments and suggestions which helped me to improve this article. The author was supported by the German Research Foundation (DFG) Project ID: 531542160 within the priority programme SPP 2265.
\bibliographystyle{abbrv}
\bibliography{references}

\end{document}